\documentclass[final]{siamart1116}

\usepackage{amsfonts}

\usepackage{bm}
\usepackage{tikz}
\usetikzlibrary{matrix,positioning,patterns,fit,plotmarks}
\usepackage{pgfplotstable}

\usepackage{amsmath} 
\usepackage{enumerate}
\usepackage{amssymb}
\usepackage{blkarray}
\usepackage{arydshln}
\usepackage{adjustbox}

\pgfplotsset{compat = 1.13,
        my ybar legend/.style={
            legend image code/.code={
                \draw [##1] (0cm,-0.6ex) rectangle +(2em,1.5ex);
            },
        },
}

\tikzset{every loop/.style={min distance=2mm,in=270,out=200,looseness=4}}
\newcommand{\Rho}{\mathrm{P}}
\newcommand{\CC}{\mathbb{C}}
\newcommand{\RR}{\mathbb{R}}
\newcommand{\FF}{\mathbb{F}}

\newcommand{\LinOp}[2]{(#2 A{-}#1 B)}
\newcommand{\LinOpInv}[2]{(#2 A{-}#1 B)^{-1}}

\newsiamthm{example}{Example} 

\numberwithin{theorem}{section}

\newcommand{\TheLongTitle}{A multishift, multipole rational QZ method with aggressive early deflation} 
\newcommand{\TheTitle}{Multishift, multipole RQZ}
\newcommand{\TheAuthors}{T. Steel, D. Camps, K. Meerbergen and R. Vandebril}

\headers{\TheTitle}{\TheAuthors}

\title{{\TheLongTitle}\thanks{Submitted to the editors March 12, 2019.
\funding{
The research was partially supported by
the Research Council KU Leuven, projects
C14/16/056 (Inverse-free Rational Krylov Methods: Theory and Applications),
OT/14/074 (Numerical algorithms for large scale matrices with uncertain coefficients).
}}}

\author{
  Thijs Steel\thanks{Department of Computer Science, KU Leuven, University of Leuven, 3001 Leuven, Belgium. (\email{thijs.steel@cs.kuleuven.be},  \email{karl.meerbergen@kuleuven.be}, \email{raf.vandebril@kuleuven.be})}
  \and
  Daan Camps\thanks{Computational Research Division, Lawrence Berkeley National Laboratory, Berkeley, CA, USA (\email{dcamps@lbl.gov})}
  \and
  Karl Meerbergen\footnotemark[2]
  \and
  Raf Vandebril\footnotemark[2]
}

\usepackage{amsopn}

\ifpdf
\hypersetup{
  pdftitle={\TheTitle},
  pdfauthor={\TheAuthors}
}
\fi

\begin{document}

\maketitle

\begin{abstract}
The rational QZ method generalizes the QZ method by implicitly supporting rational
subspace iteration.  In this paper we extend the rational QZ method by introducing shifts
and poles of higher multiplicity in the Hessenberg pencil, which is a pencil consisting of
two Hessenberg matrices.  The result is a multishift,
multipole iteration on block Hessenberg pencils which allows one to stick to real
arithmetic for a real input pencil.  In combination with optimally packed shifts and
aggressive early deflation as an advanced deflation technique we obtain an efficient
method for the dense generalized eigenvalue problem.  In the numerical experiments we
compare the results with state-of-the-art routines for the generalized eigenvalue problem
and show that we are competitive in terms of speed and accuracy.
\end{abstract}

\begin{keywords}
  generalized eigenvalues, rational QZ, rational Krylov, multishift, multipole, aggressive
  early deflation, Fortran implementation
\end{keywords}

\begin{AMS}
  65F15, 15A18
\end{AMS}

\section{Introduction}
\label{sec:introduction}

The rational QZ method (RQZ) \cite{camps2019rational} generalizes the
standard QZ method of Moler \& Stewart \cite{Moler1973}.
Both are methods for the numerical solution of the dense,
unsymmetric generalized eigenvalue problem for  a pencil of
matrices
$A, B \in \FF^{n{\times}n}$, 
$\FF \in \lbrace \CC, \RR \rbrace$.
The set of eigenvalues of the pencil $(A,B)$ is denoted as $\Lambda$
and defined by,
\begin{equation}
\label{eq:GEP}
\Lambda = \lbrace \lambda 
= \alpha/\beta \in \bar{\CC}: \det(\beta A - \alpha B) = 0 \rbrace,
\end{equation}
with $\bar{\CC} = \CC \cup \lbrace \infty \rbrace$. 
Eigenvalues with $\beta = 0$ are located at $\infty$.
We assume throughout this paper that the pencil $(A,B)$ is
\emph{regular} which means that its characteristic polynomial
differs from zero. This implies that there are exactly $n$ eigenvalues 
including those at $\infty$.

The RQZ method acts on pencils in Hessenberg, Hessenberg form instead of the Hessenberg,
triangular form used in the QZ method. For simplicity we will write \emph{Hessenberg pencil}
when we mean a pencil with two matrices in Hessenberg form. This is an advantage if, e.g.,
the input is already in this form, like in rational Krylov methods.  It relies on
\emph{pole swapping} instead of \emph{bulge chasing}.  Pole swapping was introduced by
Berljafa and G\"uttel \cite{BeGu15} to reorder the poles in rational functions used to
generate Krylov subspaces, this is done by acting on the associated Hessenberg pencil. This technique generalizes the implicitly restarted Arnoldi approach
of Sorensen \cite{So92} to the rational Krylov setting \cite{camps2019rational}. Bulge
chasing can be considered as a special instance of pole swapping: all poles are taken
equal to infinity, allowing for an alternative implementation where bulges are chased
rather than poles are swapped.  A thorough analysis of the differences between bulge
chasing and pole swapping, and the advantages and increased flexibility of the pole
swapping approach can be found in Camps, Watkins, Vandebril, and Mach
\cite{CaVaWaMa20}. We elude a little on the similarities and differences in
\Cref{sec:ssrqz}. It is, for instance, also illustrated by Camps et al.\
\cite{CaVaWaMa20}, that the natural manner of getting optimally packed poles in the pole
swapping setting is identical to the technique of Karlsson, Kressner, and Lang
\cite{KaKrLa14} of delaying some transformations to pack bulges as close as possible
together.

The RQZ method \cite{camps2019rational} computes the generalized Schur form of $(A,B)$, namely
\begin{equation}
\label{eq:QZequivalence}
(S,T) = Q^* \, (A,B) Z,
\end{equation}
such that $(S,T)$ is a triangular, triangular pencil unitarily equivalent to
$(A,B)$. The eigenvalues of $(A,B)$ are readily available as the
ratios $s_{ii}/t_{ii}$ of the diagonal elements of $S$ and $T$.
Both the single shift RQZ method and the RQZ method with optimally packed
shifts, as formulated in \cite{camps2019rational}, are applicable
to real- and complex-valued pencils. However, it requires complex arithmetic
for real-valued pencils having complex conjugate eigenvalues.

In this paper we introduce the multishift, multipole RQZ method which acts on pencils
where both matrices are in \emph{block Hessenberg} form. Though, the theory, as presented
here holds for multishifts of any size, we have only implemented a Fortran version for
multishifts and multipoles of size $2$. We do so because larger numbers of shifts
typically lead to shift blurring \cite{Watkins1996}, with delayed convergence as a
consequence.  The main benefit of using shifts and poles of higher multiplicity ($2$ is
sufficient) is that complex conjugate pairs of shifts and poles can be represented in real
arithmetic for real-valued pencils. This is similar to the well-known implicit
double-shift QR step introduced by Francis \cite{Fra62} and the double-shift QZ step
\cite{Moler1973}. The main advantages of the results in this paper are thus for the case
$\FF = \RR$.  The multishift, multipole RQZ method no longer converges to a triangular,
triangular pencil \eqref{eq:QZequivalence}; Instead, for
$A,B \in \RR^{n \times n},$ it will converge to the real generalized Schur form,
\begin{equation}
\label{eq:realQZequivalence}
(S,T) 
= Q^T \, (A,B) Z
= \left(
\begin{bmatrix}
S_{11} & S_{12} & \hdots & S_{1m} \\
0 & S_{22} & \ddots & S_{2m} \\
\vdots & \ddots & \ddots & \vdots \\
0 & \hdots & 0 & S_{mm}
\end{bmatrix},
\begin{bmatrix}
T_{11} & T_{12} & \hdots & T_{1m} \\
0 & T_{22} & \ddots & T_{2m} \\
\vdots & \ddots & \ddots & \vdots \\
0 & \hdots & 0 & T_{mm}
\end{bmatrix}
\right),
\end{equation}
where the diagonal subpencils $(S_{ii},T_{ii})$, $i=1,\hdots,m$ are of dimension
$1{\times}1$ and $2{\times}2$ and correspond to respectively the real and complex
conjugate eigenvalues of $(A,B)$.

The paper is organized as follows.
In \Cref{sec:ssrqz} we briefly refresh the rational QZ algorithm.
In \Cref{ssec:bHessDef} we define the block Hessenberg pencils on which we will be running our
algorithm. We will also characterize essential uniqueness of these block Hessenberg
pencils in \Cref{thm:implicitQ} by linking block Hessenberg pencils via unitary
equivalences to Hessenberg pencils.
The basic algorithm, where blocks containing the shifts or poles are being swapped along the
diagonal, is introduced in \Cref{ssec:manipulating}. 
\Cref{sec:deflation} discusses
aggressive early deflation for the multishift algorithm. Some particular properties, such as numerically
reliable swapping, deflation heuristics, some parameters for the blocking, and exploiting initial sparsity of the algorithm are
presented in \Cref{sec:numerical}, after which we present the full algorithm. 
Numerical
experiments are shown in \Cref{sec:numerics}. The algorithm is implemented as part of the Fortran package \texttt{libRQZ} which
is made publicly available at \url{numa.cs.kuleuven.be/software/rqz} and
\url{github.com/thijssteel/multishift-multipole-rqz}.
The convergence and the link with rational Krylov is examined in \Cref{sec:theory}.
We conclude the paper in \Cref{sec:conclusion}.

\subsection*{Notation} 
With capital letters $A, B, \hdots$ we denote matrices, $\bm{a}, \bm{b}, \hdots$ are vectors,
and $\alpha, \beta, \hdots$ scalars. Subspaces are denoted with calligraphic letters.
For example, $\mathcal{R}(A) = \mathcal{R}(\bm{a}_1,\hdots,\bm{a}_n)$ is the column space of 
$A = [ \bm{a}_1 \hdots \bm{a}_n ]$, $\mathcal{E}_k = \mathcal{R}(\bm{e}_1, \hdots, \bm{e}_k)$
with $\bm{e}_j$ the $j$th canonical basis vector of appropriate dimension. The $k$th
order polynomial Krylov subspace generated by $A$ from starting vector $\bm{v}$ is 
$\mathcal{K}_k(A,\bm{v}) = \mathcal{R}(\bm{v}, A\bm{v}, \hdots, A^{k-1}\bm{v})$.
The tuple $(\alpha, \beta)$ is ordered and $\lbrace \alpha, \beta \rbrace$ denotes an
unordered multiset with repetition; we need this to identify the poles in the block
Hessenberg pencils.
We have already used the complex plane extended with the point at infinity,
$\bar{\CC} = \CC \cup \lbrace \infty \rbrace$ and division of a
nonzero $\alpha \in \CC$ by $0$ results in infinity.


\section{The single shift RQZ algorithm}
\label{sec:ssrqz}
We summarize the flow of the QZ - \cite{b333} and the single shift rational QZ algorithm
\cite{camps2019rational}; identify similarities and differences.

The conventional QZ algorithm is
based on bulge chasing.
To initiate a QZ step, a cleverly chosen unitary transformation $Q$
is applied to $A$ and $B$.
This transformation perturbs the
Hessenberg, triangular structure creating what is called the bulge \cite{b333}. The next
steps of the algorithm consist of restoring the structure by applying unitary equivalence
transformations that chase the bulge downwards. When the bulge slides off the pencil, a 
step is completed. If the shifts are chosen well, a deflatable eigenvalue  appears in
the bottom-right corner after a few QZ steps. 
In understanding  the QZ algorithm we exploit the link between Hessenberg, triangular
pencils and polynomial Krylov subspaces.

The RQZ algorithm is an extension of the QZ algorithm that acts on Hessenberg
pencils instead;  here there are rational Krylov subspaces lurking behind.
To define a rational Krylov subspace a rational function with a prescribed set of poles is
required; these poles are encoded as ratios of the subdiagonal elements in the
Hessenberg
pencil\footnote{Reconsidering a Hessenberg, triangular
pencil, we note that all ratios will be infinity, hence the rational function becomes
polynomial and we are in the QZ case.}.
More details on the connection with
Krylov and an analysis of the theory for the multishift multipole setting discussed in
this paper are found in \Cref{sec:theory}; The single shift rational QZ is discussed by Camps
et al.\ \cite{camps2019rational}.


Consider now a Hessenberg pencil. 
Instead of introducing a bulge, we replace the pole (ratio of first subdiagonal elements)
by another pole;  if the context requires it, we name this pole the shift-pole to
emphasize where it comes from.
Next, we chase the shift-pole downward by continuously swapping the shift-pole with the pole next
to it thereby moving it to the bottom of the pencil\footnote{This technique is equivalent to reordering a
generalized Schur factorization.}. Once at the bottom of the pencil, the shift-pole is
replaced with another pole and we are done. Again, well chosen shift-poles lead to convergence.

An advantage over classical bulge chasing algorithms is that by chasing the shift-pole 
to the bottom, all the other poles will have moved up one position. Although it can take
many iterations, the poles we set at the bottom will eventually arrive at the top. This
upward movement  also drives convergence at the top of the pencil governed by the
poles. Special cases, such as the possibility to stop in the middle of a chase (we still end up with
a Hessenberg pencil) and intriguing connections between tightly packed poles and bulges
are discussed by Camps et al. \cite{CaVaWaMa20};
Mach et al. \cite{MaStVaWa20} use bidirectional chasing  to solve
palindromic eigenvalue problems by pole swapping.

\section{Block Hessenberg pencils}
\label{sec:bHessenberg}
\label{ssec:bHessDef}

In this section we define block Hessenberg pencils.
We also specify properness, since proper block Hessenberg pencils serve as
input to the rational multishift QZ algorithm.
We do not impose constraints on the sizes of the blocks, but the accompanying software operates on
blocks of size two at most, as this allows us to stick to real arithmetic for
real pencils. We can, however, not exclude that some applications might lead
to a block Hessenberg pencil in a more general form, suitable for the
rational multishift QZ algorithm instead of running an initial reduction.


\subsection{Definitions and elementary results}

%
\begin{definition}[Block triangular matrix]
\label{def:blocktriu}
A matrix $R \in \FF^{n{\times}n}$ is called a block triangular matrix\footnote{As we will
  only use upper triangular matrices, we omit the word \emph{upper} for brevity.}
with  partition $\bm{s} = (s_1, \hdots, s_m)$, $s_1{+}\hdots{+}s_m = n$,
if it admits the form,
\begin{equation}
\label{eq:blocktriu}
\begin{bmatrix}
R_{11} & R_{12} & \hdots & R_{1m} \\
       & R_{22} & \hdots & R_{2m} \\
       &        & \ddots & \vdots \\
       &        &        & R_{mm}
\end{bmatrix},
\end{equation}
with block $R_{jk}$ of size $ s_j{\times}s_k$ for $1 \leq j \leq k \leq m$.
The vector $\bm{s}$ defines the sizes of the blocks and is called the partition vector.
The  partition can be explicitly denoted as
$R_{\bm{s}}$.
\end{definition}
We also use the notation $D_{\bm{s}} = \text{diag}(D_{11}, D_{22}, \hdots, D_{mm})$
for block diagonal matrices.
Further, notice that if $R_{\bm{s}}$ is a nonsingular block  triangular matrix,
$\hat{R}_{\bm{s}} = R_{\bm{s}}^{-1}$ is also block 
triangular with an identical  partition $\bm{s}$.


The partition of a block Hessenberg matrix is defined by the partition of the block
triangular matrix in its lower left corner.
\begin{definition}[Block Hessenberg matrix]
\label{def:blockHess}
A matrix $H \in \FF^{n{\times}n}$ is called a block  Hessenberg matrix
with  partition $\bm{s} = (s_1, \hdots, s_m)$, $s_1{+}\hdots{+}s_m = n{-}1$,
if it admits the form,
\begin{equation}
\label{eq:blockHess}
H_{\bm{s}} = 
\begin{bmatrix}
\bm{h}_{11}^T & h_{12} \\
H_{21} & \bm{h}_{22}
\end{bmatrix},
\end{equation}
with $H_{21}$ an $(n{-}1){\times}(n{-}1)$ block triangular matrix with 
partition $\bm{s}$, $\bm{h}_{11}$ and $\bm{h}_{22}$ vectors of length $n{-}1$, and $h_{12}$ a scalar.
\end{definition}

 
\begin{definition}[Block Hessenberg pencil]
\label{def:blockHesspencil}
The $n{\times}n$  pencil $(A,B)$ is a block  Hessenberg pencil
with  partition $\bm{s} = (s_1, \hdots, s_m)$ if both
\begin{equation}
\label{eq:blockHesspencil}
A = 
\begin{bmatrix}
\bm{a}_{11}^T & a_{12} \\
A_{21} & \bm{a}_{22}
\end{bmatrix}\quad
\mbox{ and }
\quad
B = 
\begin{bmatrix}
\bm{b}_{11}^T & b_{12} \\
B_{21} & \bm{b}_{22}
\end{bmatrix},
\end{equation}
are block   Hessenberg matrices with a coinciding  partition.
The block  triangular pencil $(A_{21}, B_{21})$ in \cref{eq:blockHesspencil} 
is called the \emph{pole pencil} of $(A,B)$.
If the pole pencil is regular, the poles $\Xi(A,B)$ are defined as
the eigenvalues of the pole pencil $\Lambda(A_{21},B_{21})$.
\end{definition}
The
ordering of the poles plays an important role in the convergence theory (see
\Cref{sec:theory}).  Hence, for
$(A_{21},B_{21})$ from \eqref{eq:blockHesspencil}
we define the pole tuple as 
\begin{equation}
\label{eq:poleset}
\Xi(A,B) = \Lambda(A_{21},B_{21}) = (\Xi^1, \hdots, \Xi^m) 
= ( \lbrace \xi^1_{1}, \hdots, \xi^1_{s_1} \rbrace, \hdots, \lbrace \xi^m_1, \hdots, \xi^m_{s_m} \rbrace ).
\end{equation}
%
This imposes no specific ordering of the poles within a
block, but the mutual ordering of the multisets of poles corresponds to the ordering of the blocks  in $(A_{21},B_{21})$.


\begin{example}
\label{ex:blockHesspencil}
The pencil $(A,B)$ is a $9{\times}9$ block
 Hessenberg pencil with  partition 
$\bm{s} = (2, 1, 3, 2)$ if it has the form:
\begin{center}
\begin{tikzpicture}[node distance=-1ex]
\matrix (mymatrix) [matrix of math nodes,left delimiter={[},right
delimiter={]}]
  {
	\times & \times & \times & \times & \times & \times & \times & \times & \times \\
	\times & \times & \times & \times & \times & \times & \times & \times & \times \\
	\times & \times & \times & \times & \times & \times & \times & \times & \times \\
	       &        & \times & \times & \times & \times & \times & \times & \times \\
		   &        &        & \times & \times & \times & \times & \times & \times \\
		   &        &        & \times & \times & \times & \times & \times & \times \\
		   &        &        & \times & \times & \times & \times & \times & \times \\
		   &        &        &        &        &        & \times & \times & \times \\
		   &        &        &        &        &        & \times & \times & \times \\
  };

\draw[fill={rgb:red,66;green,124;blue,244}, fill opacity=0.2,rounded corners=3pt] (mymatrix-2-1.north west) rectangle (mymatrix-9-8.south east); 
\draw[dashed] (mymatrix-1-1.north west) rectangle (mymatrix-1-2.south east);
\draw[dashed] (mymatrix-1-3.north west) rectangle (mymatrix-1-3.south east);
\draw[dashed] (mymatrix-1-4.north west) rectangle (mymatrix-1-6.south east);
\draw[dashed] (mymatrix-1-7.north west) rectangle (mymatrix-1-8.south east);
\draw[dashed] (mymatrix-1-9.north west) rectangle (mymatrix-1-9.south east);

\draw[dotted] (mymatrix-2-1.north west) rectangle (mymatrix-3-2.south east);
\draw[dotted] (mymatrix-2-3.north west) rectangle (mymatrix-3-3.south east);
\draw[dotted] (mymatrix-2-4.north west) rectangle (mymatrix-3-6.south east);
\draw[dotted] (mymatrix-2-7.north west) rectangle (mymatrix-3-8.south east);
\draw[dashed] (mymatrix-2-9.north west) rectangle (mymatrix-3-9.south east);

\draw[dotted] (mymatrix-4-3.north west) rectangle (mymatrix-4-3.south east);
\draw[dotted] (mymatrix-4-4.north west) rectangle (mymatrix-4-6.south east);
\draw[dotted] (mymatrix-4-7.north west) rectangle (mymatrix-4-8.south east);
\draw[dashed] (mymatrix-4-9.north west) rectangle (mymatrix-4-9.south east);

\draw[dotted] (mymatrix-5-4.north west) rectangle (mymatrix-7-6.south east);
\draw[dotted] (mymatrix-5-7.north west) rectangle (mymatrix-7-8.south east);
\draw[dashed] (mymatrix-5-9.north west) rectangle (mymatrix-7-9.south east);

\draw[dotted] (mymatrix-8-7.north west) rectangle (mymatrix-9-8.south east);
\draw[dashed] (mymatrix-8-9.north west) rectangle (mymatrix-9-9.south east);

\matrix (mymatrixB) [matrix of math nodes,left delimiter={[},right
delimiter={]},right=1.2cm of mymatrix]
  {
	\times & \times & \times & \times & \times & \times & \times & \times & \times \\
	\times & \times & \times & \times & \times & \times & \times & \times & \times \\
	\times & \times & \times & \times & \times & \times & \times & \times & \times \\
	       &        & \times & \times & \times & \times & \times & \times & \times \\
		   &        &        & \times & \times & \times & \times & \times & \times \\
		   &        &        &  & \times & \times & \times & \times & \times \\
		   &        &        &  & \times & \times & \times & \times & \times \\
		   &        &        &        &        &        & \times & \times & \times \\
		   &        &        &        &        &        & \times & \times & \times \\
  };

\draw[fill={rgb:red,66;green,124;blue,244}, fill opacity=0.2,rounded corners=3pt] (mymatrixB-2-1.north west) rectangle (mymatrixB-9-8.south east); 
\draw[dashed] (mymatrixB-1-1.north west) rectangle (mymatrixB-1-2.south east);
\draw[dashed] (mymatrixB-1-3.north west) rectangle (mymatrixB-1-3.south east);
\draw[dashed] (mymatrixB-1-4.north west) rectangle (mymatrixB-1-6.south east);
\draw[dashed] (mymatrixB-1-7.north west) rectangle (mymatrixB-1-8.south east);
\draw[dashed] (mymatrixB-1-9.north west) rectangle (mymatrixB-1-9.south east);

\draw[dotted] (mymatrixB-2-1.north west) rectangle (mymatrixB-3-2.south east);
\draw[dotted] (mymatrixB-2-3.north west) rectangle (mymatrixB-3-3.south east);
\draw[dotted] (mymatrixB-2-4.north west) rectangle (mymatrixB-3-6.south east);
\draw[dotted] (mymatrixB-2-7.north west) rectangle (mymatrixB-3-8.south east);
\draw[dashed] (mymatrixB-2-9.north west) rectangle (mymatrixB-3-9.south east);

\draw[dotted] (mymatrixB-4-3.north west) rectangle (mymatrixB-4-3.south east);
\draw[dotted] (mymatrixB-4-4.north west) rectangle (mymatrixB-4-6.south east);
\draw[dotted] (mymatrixB-4-7.north west) rectangle (mymatrixB-4-8.south east);
\draw[dashed] (mymatrixB-4-9.north west) rectangle (mymatrixB-4-9.south east);

\draw[dotted] (mymatrixB-5-4.north west) rectangle (mymatrixB-7-6.south east);
\draw[dotted] (mymatrixB-5-7.north west) rectangle (mymatrixB-7-8.south east);
\draw[dashed] (mymatrixB-5-9.north west) rectangle (mymatrixB-7-9.south east);

\draw[dotted] (mymatrixB-8-7.north west) rectangle (mymatrixB-9-8.south east);
\draw[dashed] (mymatrixB-8-9.north west) rectangle (mymatrixB-9-9.south east);

\node[] at (2.8,0) {,};
\node[] at (8.5,0) {.};
 \end{tikzpicture}
\end{center}
The shaded part of the matrices is the pole pencil which is in
the desired block triangular form.
Note that matrix $B$ also
allows a partition $(2,1,1,2,2)$; but, we require a matching partition for both $A$ and $B$
and hence the third block ($1\times 1$) and the fourth block ($2\times 2$)  are merged to
form a single block of size $3\times 3$.
 The pole tuple is
$\Xi(A,B) = (\lbrace \xi_1^1,\xi_2^1 \rbrace, \; \lbrace \xi_1^2 \rbrace, \;
 \lbrace \xi_1^3, \xi_2^3, \xi_3^3 \rbrace, \; \lbrace \xi_1^4, \xi_2^4 \rbrace )$, 
where $\xi_1^1$ and $\xi_2^1 $ are the eigenvalues of the principal $2\times 2$
block, and so forth.
\end{example}


Following \Cref{ex:blockHesspencil}, we note that block Hessenberg matrices and pencils can admit more than one partition. 
If $(A,B)$ is a block Hessenberg pencil with partition
$\bm{s} = (s_1, \hdots, s_k, s_{k+1}, \hdots, s_m)$,
then it also admits the partition
$\hat{\bm{s}} = (s_1, \hdots, s_k + s_{k+1}, \hdots, s_m)$; 
consecutive blocks can be grouped together.

We will prove in \Cref{sec:theory} that convergence
will only take place in-between blocks. It is therefore better to have as many blocks
as possible, meaning that none of the blocks should be splittable into smaller blocks.
To stick to real arithmetic in the real case, we can restrict ourselves to
blocks of sizes $1\times 1$ and $2\times 2$.


\subsection{Properness and uniqueness}

The final definition generalizes
\emph{properness} or \emph{irreducibility} to a block Hessenberg pencil.
Being proper guarantees that there are no obvious options for a deflation 
that splits the problem into smaller, independent problems.

\begin{definition}
\label{def:properness}
A regular block  Hessenberg pencil $(A,B)$ with partition
$\bm{s} = (s_1, \hdots, s_m)$ is said to be proper (or irreducible) if:
\begin{enumerate}[I.]
\item Its pole pencil is regular;

\item The first block column of $(A,B)$ of size $(s_1{+}1){\times}s_1$,
\begin{equation*}
\begin{bmatrix}
\bm{{a}}_{1,1}^T \\
{A}_{2,1}
\end{bmatrix}
= 
\begin{bmatrix}
\bm{a}_1 & \hdots & \bm{a}_{s_1}
\end{bmatrix},
\quad
\begin{bmatrix}
\bm{{b}}_{1,1}^T \\
{B}_{2,1}
\end{bmatrix}
= 
\begin{bmatrix}
\bm{b}_1 & \hdots & \bm{b}_{s_1}
\end{bmatrix},
\quad \bm{a}_i, \bm{b}_i \in \FF^{s_1+1},
\end{equation*}
satisfies for $i=1, \hdots, s_1$,
\begin{equation*}
\mathcal{R}(\bm{a}_1, \hdots, \bm{a}_i)
\neq
\mathcal{R}(\bm{b}_1, \hdots, \bm{b}_i);
\end{equation*}

\item The last block row of $(A,B)$ of size $s_m{\times}(s_m{+}1)$,
\begin{equation*}
\resizebox{.9 \textwidth}{!} 
{$
\begin{bmatrix}
A_{m+1,m} & \bm{a}_{m+1,m+1}
\end{bmatrix}
=
\begin{bmatrix}
\bm{a}_{s_m}^T \\ \vdots \\ \bm{a}_{1}^T
\end{bmatrix},
\quad
\begin{bmatrix}
B_{m+1,m} & \bm{b}_{m+1,m+1}
\end{bmatrix}
=
\begin{bmatrix}
\bm{b}_{s_m}^T \\ \vdots \\ \bm{b}_{1}^T
\end{bmatrix},
\quad \bm{a}_i, \bm{b}_i \in \FF^{s_m+1},
$}
\end{equation*}
satisfies for $i=1, \hdots, s_m$,
\begin{equation*}
\mathcal{R}(\bm{a}_1, \hdots, \bm{a}_i)
\neq
\mathcal{R}(\bm{b}_1, \hdots, \bm{b}_i).
\end{equation*}
\end{enumerate}
\end{definition}

Violating conditions II or III implies a deflation at either the top-left or bottom-right
corner of the pencil. Condition I ensures there are no obvious deflations anywhere in the
interior of the pencil.
We illustrate \Cref{def:properness} by an example.

\begin{example}
Consider the $4{\times}4$ real-valued block Hessenberg pencil $(A,B)$ with partition
$(2,1)$ given by:
\begin{equation*}
\left[\begin{array}{llll}
-0.300  & 0.075 & \phantom{-}0.500   & \phantom{-}0.250 \\
\phantom{-}0.395 & 0.520  & -0.350 & \phantom{-}2.000    \\
-0.140 & 0.860  & \phantom{-}1.350  & -0.800 \\
      &       &  \phantom{-}1.000   & \phantom{-}0.850
\end{array}\right], \qquad
\begin{bmatrix}
-0.150 & -0.600  & \phantom{-}0.150 & -1.500 \\
\phantom{-}0.160  & \phantom{-}0.940  & -5.000   & \phantom{-} 1.350 \\
-0.120 & -0.080 & -2.400 & -1.000   \\
      &       & \phantom{-} 0.200 & \phantom{-} 1.800
\end{bmatrix}.
\end{equation*}
Condition~I of \Cref{def:properness} is satisfied, the pole pencil is regular and the
pole tuple of $(A,B)$ is given by:
\begin{equation}
\label{eq:polesimproperpencil}
\Xi = (\lbrace 1.5 + i\sqrt{15/8}, 1.5 - i\sqrt{15/8} \rbrace, 5).
\end{equation}
Condition~III of \Cref{def:properness} is also satisfied.
For the last block row of $(A,B)$, we clearly have that
$\mathcal{R}(\begin{bmatrix}
1 & 0.85
\end{bmatrix}^T)
\neq
\mathcal{R}(\begin{bmatrix}
0.2 & 1.8
\end{bmatrix}^T)
$.
Notice that this implies that we cannot simultaneously create a zero
in position $(4,3)$ of both $A$ and $B$ by rotating the last two columns. 
The pencil is, however, \emph{improper}
as Condition~II of \Cref{def:properness} is violated.
We have $\mathcal{R}(\bm{a}_1) \neq \mathcal{R}(\bm{b}_1)$, but
$\mathcal{R}(\bm{a}_1,\bm{a}_2) = \mathcal{R}(\bm{b}_1,\bm{b}_2)$.
If we compute an orthonormal basis $Q_1$ of $\mathcal{R}(\bm{a}_1,\bm{a}_2)$
and extend this upto an orthonormal matrix $Q = \begin{bmatrix} Q_1 & \bm{q}_2\end{bmatrix}$,
then $(\hat{A},\hat{B}) = Q^T (A,B)$ has zero elements in positions $(3,1)$
and $(3,2)$.
This allows us to split the problem into two submatrices and deflate 
the eigenvalues of the leading $2\times 2$ subpencil, which are the complex conjugate  poles in
\cref{eq:polesimproperpencil}.
\end{example}

The next theorem generalizes the implicit Q theorem to block Hessenberg pencils. We
define first what \emph{essentially unique} means in the block formulation.

\begin{definition}
  \label{defi:esu}
  Two block Hessenberg pencils $(\hat{A},\hat{B})$ and $(\check{A},\check{B})$ with
  partitions $\hat{\bm{s}}$ and $\check{\bm{s}}$ respectively are said to
  be essentially identical if there exist  equivalences  with unitary block
  diagonal matrices such that
\begin{equation*}
  \hat{D}_1^* (\hat{A},\hat{B}) \hat{D}_2 = \check{D}^*_1 (\check{A},\check{B}) \check{D}_2,
\end{equation*}
where $\hat{D}_1$ has partition $(1,\hat{\bm{s}})$ and $\hat{D}_2$ has partition
$(\hat{\bm{s}},1)$; and $\check{D}_1$ and $\check{D}_2$ have partitions
$(1,\check{\bm{s}})$ and $(\check{\bm{s}},1)$.
\end{definition}

Though not mentioned explicitly in the definition, we will see in the proof of
\Cref{thm:implicitQ} that one can transform two essentially identical block Hessenberg pencils
to the same Hessenberg pencil. Moreover, suppose that we have a real block Hessenberg pencil then the
pencil will always be essentially identical to a real block Hessenberg pencil having the
$B$ matrix of Hessenberg form; this is the form we will use in our implementation.


To simplify the formulation of the next theorem we clarify the wording \emph{extracting a
list of poles}.
The pole tuple is a list of multisets; with extracting a list of poles out of the pole
tuple we mean assigning an ordering to elements in the multisets. Suppose, as an example,
that we have a pole tuple $(\{3,2,1\},4,\{6,5\},\{8,7\})$, then we can extract various
lists out of it, such as, e.g., $(3,2,1,4,6,5,8,7)$, but also $(1,3,2,4,5,6,7,8)$. There
are typically many possibilities of extracting a list of poles out of a pole tuple.
Suppose the pencil $(\hat{A},\hat{B})$ has pole set $(\{3,2,1\},4,\{6,5\},\{8,7\})$ and
$(\check{A},\check{B})$ has pole set $(1,\{5,4,3,2\},\{7,6\},8)$.
Then we can extract a list of poles $(1,2,3,4,5,6,7,8)$ that is feasible for both
pencils, this is required in the next theorem.

\begin{theorem}
\label{thm:implicitQ}
Let $(A,B)$ be a  proper matrix pencil and let $\hat{Q}, \check{Q}, \hat{Z}, \check{Z}$
be unitary matrices with $\hat{Q}\bm{e}_1 = \sigma \check{Q} \bm{e}_1$, $|\sigma| = 1$,
such that,
\begin{equation*}
(\hat{A}, \hat{B}) = \hat{Q}^* (A,B) \hat{Z}
\quad \text{and} \quad
(\check{A}, \check{B}) = \check{Q}^* (A,B) \check{Z}
\end{equation*}
are block Hessenberg pencils.
If we can extract an identical list of poles out of the pole tuples of both pencils, then 
the pencils $(\hat{A}, \hat{B})$ and $(\check{A},\check{B})$ are essentially unique.
\end{theorem}

\begin{proof} Consider the generalized Schur factorization of each individual diagonal block in the pole
   pencils of $(\hat{A},\hat{B})$ and $(\check{A},\check{B})$, where we have ordered the
   eigenvalues in the resulting Schur forms in the same manner for both pencils. We know,
   because of the assumptions on the poles in
   the theorem, that such an ordering must exist.

   Since the pole pencils are of size $(n-1)\times (n-1)$, we must embed the unitary
   transformations in an $n\times n$ identity matrix before we can apply them on the
   pencils directly.  As a result we get a multiplication with block diagonal unitary
   matrices resulting in two Hessenberg pencils with identical pole tuples
   \begin{equation*}
      \hat{D}^*_1 (\hat{A},\hat{B}) \hat{D}_2
      \quad \text{and} \quad
      \check{D}^*_1 (\check{A},\check{B}) \check{D}_2, 
   \end{equation*}
   where, because of the embedding $\hat{D}^*_1 \bm{e}_1= |\hat{\sigma}| \bm{e}_1$,
   $\check{D}^*_1 \bm{e}_1= |\check{\sigma}| \bm{e}_1$, $\hat{D}_2 \bm{e}_n=
   |\hat{\gamma}| \bm{e}_n$, and $\check{D}_2 \bm{e}_n= |\check{\gamma}| \bm{e}_n$,
   and $\hat{\sigma}$, $\check{\sigma}$, $\hat{\gamma}$, and $\check{\gamma}$ are
   unimodular.
   
   It is easy to verify that both pencils will be proper and looking back at the original pencil $(A,B)$
   we obtain $\hat{D}^*_1 \hat{Q}^* (A,B) \hat{Z} \hat{D}_2$
   and
   $ \check{D}^*_1 \check{Q}^* (A,B) \check{Z}  \check{D}_2 $
  on which we can apply the  rational implicit Q theorem \cite[Theorem
  5.1]{camps2019rational}. As a result both 
 Hessenberg pencils are essentially unique, thereby proving  the theorem.
\end{proof}


This theorem does not only justify our definition of essential uniqueness, it also provides
a manner of mapping block Hessenberg pencils to Hessenberg pencils. This mapping is not
unique, it is essentially unique in the block sense. It should thus be clear that results that hold for
the Hessenberg case, such as convergence, still hold, but only in-between blocks. We can not
make claims of what is happening within a block.



\section{The basic algorithm}
\label{ssec:manipulating}

The algorithm operates on proper pencils, otherwise it will break down;
if the pencil is improper, deflation is possible, and the problem should be split into
independent subproblems.  We assume to be operating on a
proper block Hessenberg pencil $(A,B)$ with partition
$\bm{s} = (s_1,\hdots,s_m)$ and pole tuple $\Xi =(\Xi^1, \hdots, \Xi^m)$,
where  $\Xi^j$  is a multiset of poles.
All poles
are assumed different from the eigenvalues.
We review two different operations to change the pole tuple $\Xi$.
The first operation replaces the first or last blocks of pole of the pencil,
the second operation swaps two adjacent pole blocks. 

Given a matrix pencil $(A, B) \in \FF^{n \times n}$ (generic, not necessarily block Hessenberg), with
$\varrho = \mu/\nu \in \bar{\CC}$ and 
$\xi = \alpha/\beta \in \bar{\CC}\setminus\Lambda$, 
we define the following elementary rational matrices,
\begin{equation}
\label{eq:elemRational}
\begin{split}
M(\varrho,\xi) = \LinOp{\mu}{\nu} \LinOpInv{\alpha}{\beta}, \\
N(\varrho,\xi) = \LinOpInv{\alpha}{\beta} \LinOp{\mu}{\nu}.
\end{split}
\end{equation}
Notice that the matrices $M(\varrho,\xi)$ and $N(\varrho,\xi)$ represent
an entire class of matrices that are all nonzero scalar multiple of each other.
Every representative is fine.

\paragraph{Changing poles at the boundary}


The first $\ell=s_1{+}\hdots{+}s_i$ poles
in the first $i$ pole blocks $\Xi^1, \hdots, \Xi^i$ can be replaced by
$\ell$ new poles $\Rho = \lbrace \varrho_1, \hdots, \varrho_{\ell} \rbrace$, which are
considered different from the original poles.
These new poles are classically called the shifts. In fact they are nothing else than
poles, but when it is necessary to emphasize their origin, we call them shift-poles.
To introduce the shift-poles consider the vector,
\begin{equation}
\label{eq:vec_multishift}
\bm{x} = \gamma \; \prod_{j=1}^{\ell} M(\varrho_j,\xi_j) \; \bm{e}_1,
\end{equation}
with $\xi_1, \hdots, \xi_{\ell}$ the poles of $\Xi^1, \hdots, \Xi^i$.
 Now compute a unitary matrix $Q$ such that,
 \begin{equation}
 \label{eq:Qmultishift}
 Q^* \bm{x} = \alpha \bm{e}_1.
 \end{equation}
 We will prove in \Cref{sec:theory} that
the new poles $P$ are introduced in the block Hessenberg pencil by updating
$(\hat{A},\hat{B}) = Q^* (A,B)$. 
More precisely we end up with    a block
Hessenberg pencil $(\hat{A},\hat{B})$ with partition $\hat{\bm{s}} = (\ell, s_{i+1}, \hdots, s_m)$ 
and poles $\hat{\Xi} = (\Rho, \Xi^{i+1}, \hdots, \Xi^{m})$.

The last $\ell$ poles, say $\xi_j$ for $j= m-\ell+1, \ldots m$, the last $i$ blocks $\Xi^{m-i+1}, \hdots, \Xi^{m}$
of $(A,B)$ can be changed to $\Rho = \lbrace \varrho_1, \hdots, \varrho_\ell \rbrace$ 
in a similar fashion.
We compute first the row vector,
\begin{equation}
\label{eq:vec_multishift_end}
\bm{x}^T = \gamma \bm{e}_{n}^T \prod_{j=m-\ell+1}^{m} N(\varrho_j,\xi_j),
\end{equation}
and then a unitary matrix $Z = \text{diag}(I,Z_{\ell+1})$ such that
$\bm{x}^T Z = \alpha \bm{e}_n^T$. The pencil $(\hat{A},\hat{B}) = (A,B) Z$ then becomes
 block Hessenberg with pole tuple $(\Xi^1$, $\hdots$, $\Xi^{m-i}$, $\Rho)$.

In order to compute the vector $\bm{x}$ of \Cref{eq:vec_multishift},
$\ell$ shifted linear systems need to be solved as 
$M(\varrho_i,\xi_i) = \LinOp{\mu_i}{\nu_i} \LinOpInv{\alpha_i}{\beta_i}$. 
These linear systems are essentially of size $\ell$ because $\LinOpInv{\alpha_\ell}{\beta_\ell}$
is a block  triangular matrix with a leading block of size $\ell{\times}\ell$.
This limits the computational cost of computing $\bm{x}$ to $O(\ell^3)$, which
is small as long as $\ell \ll n$.
It also follows that the vector $\bm{x}$ can be computed even when poles in 
$\Xi^{1}$,$\hdots$, $\Xi^{i}$ are equal to eigenvalues of the pencil not present in the
leading block, since properness ensures this.

We remark that if $(A,B)$ is a real-valued pencil and the poles and shift-poles considered 
in \cref{eq:vec_multishift,eq:vec_multishift_end} are both closed under complex conjugation,
then the vectors $\bm{x}$ and $\bm{x}^T$ and consequently the matrices $Q$ and $Z$ are also
real-valued. This follows from the commutativity of the elementary rational matrices $M(\varrho,\xi)$ and $M(\bar{\varrho},\bar{\xi})$  (see \cite{camps2019rational}) in combination with the property
that $M(\bar{\varrho},\bar{\xi}) = \overline{M(\varrho,\xi)}$ for real-valued pencils.
We have,
\begin{equation}
\label{eq:realvalued}
\overline{M(\varrho,\xi) M(\bar{\varrho},\bar{\xi})} =
\overline{M(\bar{\varrho},\bar{\xi}) M(\varrho,\xi)} =
M(\varrho,\xi) M(\bar{\varrho},\bar{\xi})
\end{equation}
so $M(\varrho,\xi) M(\bar{\varrho},\bar{\xi})$ is a real-valued matrix if $A$ and $B$ are real-valued.

\paragraph{Swapping adjacent pole blocks}

Swapping block $i$ with block $i{+}1$ requires the computation of a unitary equivalence
 of size $(s_i{+}s_{i+1}){\times}(s_i{+}s_{i+1})$ that updates
the pencil $(\hat{A},\hat{B}) = Q^* (A,B) Z$ in such a way
that the new pole tuple and partition vector are given by,
\begin{equation*}
\begin{split}
\hat{\Xi} &= (\Xi^1, \hdots, \Xi^{i-1}, \Xi^{i+1}, \Xi^{i}, \Xi^{i+2}, \hdots, \Xi^m), \\
\hat{\bm{s}} &= (s_1, \hdots, s_{i-1}, s_{i+1}, s_{i}, s_{i+2}, \hdots, s_m).
\end{split}
\end{equation*}
This problem is equivalent to reordering eigenvalues in the generalized
Schur form. 
Two different approaches to solve this problem have been proposed in the literature.
The first approach, studied by K{\aa}gstr{\"{o}}m \cite{Kagstrom1993,Kagstrom1996},
requires the solution of a coupled Sylvester equation.
This method is applicable for general blocksizes.
The second approach, studied by Van Dooren \cite{VanDooren1981}, is a direct method
that relies on the computation of a right eigenvector of a pole in block $i{+}1$.
This method has been studied for swapping a block of dimension $1{\times}1$ or $2{\times}2$
with a block of dimension $1{\times}1$, or vice versa. In our implementation we combine both
techniques and use iterative refinement. More details are discussed in \Cref{sec:numerical}.

\paragraph{Multishift, multipole RQZ step}
\label{ssec:msmpstep}

We propose the following three step
procedure as the generic multishift, multipole RQZ step. 

\begin{enumerate}[I.]

\item \label{step:MSMPRQZ1} 
Starting from a proper block Hessenberg pencil with pole tuple
$\Xi = (\Xi^1, \hdots, \Xi^m)$ and partition
$\bm{s} = (s_1, \hdots, s_m)$. Select or compute 
$\ell$ shifts $\Rho$.
Introduce the shift-poles in the block Hessenberg pencil.
The pencil now has pole tuple $\Xi = (\Rho, \Xi^{i+1}, \hdots, \Xi^m)$
and partition vector $\bm{s} = (\ell, s_{i+1}, \hdots, s_m)$.

\item \label{step:MSMPRQZ2}
Repeatedly use the swapping procedure to construct unitary
equivalences that swap the block carrying the shift-poles  with the next pole-block in line.
In the end the shift-poles reach the last block and the pole tuple equals $\Xi = (\Xi^{i+1}, \hdots, \Xi^m,\Rho)$
and the partition vector  $\bm{s} = (s_{i+1}, \hdots, s_m, \ell)$.

\item \label{step:MSMPRQZ3}
Compute or select $\ell$ new poles $\Xi^{m+1}$ and introduce them
at the end of the pencil to change the pole tuple to $\Xi = (\Xi^{i+1}, \hdots, \Xi^m,\Xi^{m+1})$.
\end{enumerate}
These three steps constitute a single multishift, multipole RQZ sweep.
After every sweep, the properness of the pencil is checked and the problem
is split into independent subproblems wherever possible. We will prove in \Cref{sec:theory}
that continuously executing RQZ sweeps with well-chosen shifts will lead to deflations and eventually
converges to the Schur form. Moreover, we stress that this is the basic form, an enhanced
form of the algorithm, with aggressive deflation (see \Cref{sec:deflation}) is discussed in \Cref{sec:heuristics}.

The multishift QZ method is a special case of this algorithm where
the pencil initially has pole tuple $(\infty, \hdots, \infty)$ 
and partition $(1,\hdots,1)$ and where this form is always restored in
step~\ref{step:MSMPRQZ3} of the algorithm by only allowing the introduction of poles at infinity.

The shift-poles will create convergence in the lower-right corner of the matrix, the poles
move up slowly and will lead to deflations in the upper-left corner of the matrix. In the
classical QZ algorithm, it is not possible to stear convergence in the upper-left
corner, the poles moving up, will always be equal to $\infty$.
Again, we defer the
theoretical analysis of the convergence to \Cref{sec:theory}.


\section{Aggressive early deflation}
\label{sec:deflation}

Aggressive early deflation (AED) significantly speeds up the convergence of the QR \cite{Braman2002a} 
and QZ \cite{Kagstrom2007} methods by identifying deflatable eigenvalues before classical
deflation criteria are able to detect them: We will check whether the eigenvalues of leading and
trailing principal subpencils are eigenvalues of the original pencil. 


Because the shift-poles lead
to convergence in the bottom-right corner of the pencil and the poles cause convergence 
in the upper-left corner, AED can be
performed at both sides of the pencil. We present the description of the AED process
only for the upper-left sides of the pencil, the bottom-right proceeds similarly. 

The deflation window sizes are $w_e$ for the bottom-right  and $w_s$ for the upper-left side of
the pencil, they are chosen to cover an integer number of blocks, avoiding thereby subdivision of blocks.
The deflation windows are shown in Pane I of \Cref{fig:aed}. The window sizes used in our
practical implementation  are found in \Cref{sec:heuristics}.

The pencil $(A,B)$ is subdivided as follows
\begin{equation}
\label{eq:aed_part1}
\begin{blockarray}{c|cc|cc}
w_s &  &  & w_e \\
\begin{block}{[c|cc|c]c}
A_{11} & A_{12} & A_{13} & A_{14} & w_s \\
\cline{1-5}
A_{21} & A_{22} & A_{23} & A_{24} &  2 \\
       & A_{32} & A_{33} & A_{34} & \\
\cline{1-5}
       &        & A_{43} & A_{44} & w_e \\
\end{block}
\end{blockarray}, \quad
\begin{blockarray}{c|cc|cc}
w_s &  &  & w_e \\
\begin{block}{[c|cc|c]c}
B_{11} & B_{12} & B_{13} & B_{14} & w_s \\
\cline{1-5}
B_{21} & B_{22} & B_{23} & B_{24} & 1 \\
       & B_{32} & B_{33} & B_{34} & \\
\cline{1-5}       
       &        & B_{43} & B_{44} & w_e \\
\end{block}
\end{blockarray}
,
\end{equation}
where we assume $A_{21}$ to have two rows, indicating that we have a $2\times 2$ block
just after the deflation window. We restrict the size of $A_{21}$ and of the other blocks
to be at most $2\times 2$ since it simplifies the presentation and it is sufficient for
the real double-shift rational QZ algorithm.  Recall from the discussion following \Cref{thm:implicitQ} that
we have taken the design decision to have $B$ in Hessenberg form, as a consequence
$B_{21}$ has only one row.
The subpencils $(A_{11},B_{11})$ and $(A_{44},B_{44})$ are the upper-left and bottom-right deflation
windows.

\begin{figure}[htp]
\label{fig:aed}
\centering
\resizebox{\textwidth}{!}{%
\input{./fig/figaed.tikz}
}
\caption{Visualization of the three stages of aggressive early deflation.
}
\end{figure}

In the first phase, shown in pane II of \Cref{fig:aed},  the parts of the pencil
within the deflation windows are reduced to quasi Schur form (block upper triangular with blocks of size
at most $2$):
\begin{equation}
\label{eq:aed_part2}
(S_{11},T_{11}) = Q^{T}_{s} (A_{11},B_{11}) Z_{s},
\quad \text{and} \quad
(S_{44},T_{44}) = Q^{T}_{e} (A_{44},B_{44}) Z_{e}.
\end{equation}
This can be done with the
RQZ method as all subpencils in the deflation windows are in block Hessenberg form.
Applying these transformations as an equivalence to $(A,B)$ results in the pencil $(\check{A},\check{B})$:
\begin{equation}
\label{eq:aed_part3}
\resizebox{.89\hsize}{!}{$
\begin{blockarray}{c|cc|c}
\begin{block}{[c|cc|c]}
S_{11}       & Q^{T}_{s} A_{12} & Q^{T}_{s} A_{13} & Q^{T}_{s} A_{14} Z_{e}\\
\cline{1-4}
A_{21} Z_{s} & A_{22}           & A_{23}           & A_{24} Z_{e}           \\
             & A_{32}           & A_{33}           & A_{34} Z_{e}       \\
\cline{1-4}
             &                  & Q^{T}_{e} A_{43} & S_{44}             \\
\end{block}
\end{blockarray},
\begin{blockarray}{c|cc|c}
\begin{block}{[c|cc|c]}
T_{11}       & Q^{T}_{s} B_{12} & Q^{T}_{s} B_{13} & Q^{T}_{s} B_{14} Z_{e} \\
\cline{1-4}
B_{21} Z_{s} & B_{22}           & B_{23}           & B_{24} Z_{e} \\
             & B_{32}           & B_{33}           & B_{34} Z_{e} \\
\cline{1-4}       
             &                  & Q^{T}_{e} B_{43} & T_{44} \\
\end{block}
\end{blockarray}.
$}
\end{equation}

The spikes shown in pane~II of \Cref{fig:aed} correspond to the matrices
$A_{21}Z_s$, $B_{21} Z_s$, $Q^{T}_{e} A_{43}$, and $Q^{T}_{e} B_{43}$, with 
$B_{21} Z_s$ of dimension $1{\times}w_s$ and
%
$A_{21}Z_s$ of dimension $2{\times}w_s$.
The two rows of $A_{21} Z_s$
are scalar multiples of each other and also a multiple of $B_{21} Z_s$.
We denote with $\bm{p}_{s}^{B} = b_{w_s+1,w_s} \bm{e}_{w_s}^T Z_{s} =
b_{w_s+1,w_s} \bm{p}_s$
the spike at the upper-left deflation window of $B$. Similarly,
$\bm{p}_{s}^{A} = \zeta \bm{e}_{w_s}^T Z_{s} = \zeta \bm{p}_s$, with $\zeta$ equal to
$|a_{w_s+1,w_s}|+|a_{w_s+2,w_s}|$. Taking the sum makes the deflation check, discussed further on, easier.

The second phase in the AED process is illustrated in Pane~III of \Cref{fig:aed} and entails
testing for deflatable eigenvalues inside the deflation windows. The deflation test starts at
the left-side of the spikes $\bm{p}_{s}^{A}$ and $\bm{p}_{s}^{B}$.
If there is a $1{\times}1$  eigenvalue
located at this position, we test if:
\begin{equation}
\label{eq:aed_part5}
| \bm{p}_{s,1}^{A} | < c \epsilon_{m} (|a_{1,1}| + |a_{2,2}|)
\quad \text{and} \quad
| \bm{p}_{s,1}^{B} | < c \epsilon_{m} (|b_{1,1}| + |b_{2,2}|),
\end{equation}
with $c$ a modest constant and $\epsilon_m$ the machine precision.
If there is a $2 \times 2$ block at this position (which means a complex conjugate pair of
eigenvalues in the real case) at this position,
we test if:
\begin{equation}
\label{eq:aed_part6}
| \bm{p}_{s,1}^{A} | + | \bm{p}_{s,2}^{A} | < c \epsilon_{m} \|A_{1:2,1:2}\|_{F}
\quad \text{and} \quad
| \bm{p}_{s,1}^{B} | + | \bm{p}_{s,2}^{B} | < c \epsilon_{m} \|B_{1:2,1:2}\|_{F}.
\end{equation}

If the first eigenvalue is deflatable according to \cref{eq:aed_part5} or
\cref{eq:aed_part6}, the corresponding spike elements in $\bm{p}_{s}^{A}$ 
and $\bm{p}_{s}^{B}$ are set to zero and the next eigenvalue is tested according
to the same criterion.
If the first eigenvalue is not deflatable, another eigenvalue that has not yet been
tested, is swapped to the top-left corner, this also changes the values of the spikes.
Then it is checked if this is deflatable according to \cref{eq:aed_part5} or
\cref{eq:aed_part6}.
This procedure is continued until all deflatable eigenvalues inside the deflation window
are identified.
The swapping of eigenvalues within the deflation window does not change the form of
\cref{eq:aed_part2} but of course $S_{11}$ and $T_{11}$ change, just like the
vector $\bm{p}_s$.
The same strategy is used for AED at the bottom-right side of the pencil.
In pane~III of \Cref{fig:aed} all spike elements that signal a deflation are marked in red.

In the third and last phase, the nonzero spike elements are removed in such a
way that the (block) Hessenberg form is restored. This form is shown
in pane~IV of \Cref{fig:aed}: the larger block in the middle is in block
Hessenberg form and the smaller blocks at the upper-left and bottom-right side
of the pencil are in quasi Schur form.
The block Hessenberg restoration is achieved by a sequence of rotations as follows.
Assume that  after all the reordering the new spike equals $\hat{\bm{p}}_s$ and  we
have the Schur form $(\hat{S}_{11}, \hat{T}_{11})$, where the first $i$ entries of
$\hat{\bm{p}}_s$ are zero. This means that the upper-left $i\times i$ block of
$(\hat{S}_{11}, \hat{T}_{11})$ is deflatable. We then
compute rotations $G_{i+1}, \hdots, G_{w_s-1}$ such that,
$\hat{\bm{p}}_s G_{i+1} \cdots G_{w_s-1}$ becomes a multiple of $\bm{e}_{w_s}^T$.
Applying those transformations to the left of $(\hat{S}_{11}, \hat{T}_{11})$ restores the
block Hessenberg form and removes the spikes.

At the end of the AED phase, we can re-initiate the chasing and use some of the undeflated eigenvalues of
$(\hat{S}_{11}, \hat{T}_{11})$ as poles, to be inserted in the bottom-right in the
upcoming rational multishift QZ step.  In
case
that AED was so succesful that many converged eigenvalues have been deflated, one could
do another step of AED before proceeding to the chase. This technique also used in our final algorithm in \Cref{sec:heuristics}.


\section{The implemented algorithm and heuristics}
\label{sec:numerical}
\label{sec:heuristics}

We have discussed the basic algorithm in \Cref{ssec:manipulating} and AED in
\Cref{sec:deflation}. Here we propose the full algorithm including additional add-ons in
the implementation such as blocking and iterative refinement for inaccurate swaps. Also
the decisions on heuristics such as deflations, block sizes, aggressive deflation windows
and the number of tightly packed poles are presented.

\subsection{Deflation} We always take $B$ in Hessenberg form (see the comments following
\Cref{{defi:esu}}); This simplifies the deflation criteria based on \Cref{def:properness}.

\paragraph{Checking Condition I of \Cref{def:properness}}
The $i$th pole along the subdiagonal is considered deflatable if,
\begin{equation}
\label{eq:defsing}
|a_{i+1,i}| < c \epsilon_m (|a_{i,i}|+|a_{i+1,i+1}|),
\quad \text{and,}  \quad
|b_{i+1,i}| < c \epsilon_m (|b_{i,i}|+|b_{i+1,i+1}|),
\end{equation}
in the case of a single pole.
If the $i$th pole is a double-pole 
we consider it deflatable if either there are small elements in the leading column of the block,
\begin{equation}
\label{eq:defdbl1}
\begin{split}
|a_{i+1,i}| + |a_{i+2,i}| & < c \epsilon_m (|a_{i,i}|+|a_{i+1,i+1}|),
\quad \text{and,}  \\
|b_{i+1,i}| & < c \epsilon_m (|b_{i,i}|+|b_{i+1,i+1}|),
\end{split}
\end{equation}
or there are sufficiently small elements in the trailing row of the block,
\begin{equation}
\label{eq:defdbl2}
\begin{split}
|a_{i+2,i}| + |a_{i+2,i+1}| & < c \epsilon_m (|a_{i+1,i+1}|+|a_{i+2,i+2}|),
\quad \text{and,}  \\
|b_{i+2,i+1}| & < c \epsilon_m (|b_{i+1,i+1}|+|b_{i+2,i+2}|).
\end{split}
\end{equation}

\paragraph{Checking Condition II and III of \Cref{def:properness}}
The first pole block of size $s_1 = 1$ or $2$ can be deflated whenever there exists an $(s_{1}+1){\times}(s_1+1)$
orthogonal matrix $Q$ such that,
\begin{equation}
\label{eq:defstart}
Q^{T}\left(
\begin{bmatrix}
\bm{a}_{1,1}^T \\
A_{2,1}
\end{bmatrix},
\begin{bmatrix}
\bm{b}_{1,1}^T \\
B_{2,1}
\end{bmatrix}
\right)
=
\left(
\begin{bmatrix}
A_{1,1} \\
\bm{0}^T
\end{bmatrix},
\begin{bmatrix}
B_{1,1} \\
\bm{0}^T
\end{bmatrix}
\right)
\end{equation}
Here, the last row is considered numerically zero according to a relative tolerance
similar to \cref{eq:defsing,eq:defdbl1,eq:defdbl2}. The matrix $Q$ is constructed such
to create a desired zero in the column with largest norm; if the pole block is deflatable, this
should create zeros in the other positions as well.
A similar approach is used to check for deflations in the last block row.

\subsection{Accuracy of swapping pole blocks}

The equivalence between swapping eigenvalues in the generalized Schur form and swapping
poles allows us to use the \texttt{LAPACK} routine \texttt{DTGEXC}. However, some modifications can greatly increase the effectiveness of this routine.

When one of the blocks to be swapped is of size 1, the method described in
\cite{VanDooren1981} can be used. We note that a modest modification proposed by Camps et
al.\ \cite{CaVaWaMa20} of Van Dooren's
method \cite{VanDooren1981} improves the theoretical bound on the backward error of the
swap.  
The computed matrices $\hat{Q}$,$\hat{Z}$,$\hat{A}$,$\hat{B}$ are proven to satisfy
\begin{equation}
   \hat{Q}^*(A+E_A,B+E_B)\hat{Z} = (\hat{A},\hat{B}),
\end{equation}
with $\|E_A\|_2 \le c\epsilon\|A\|_2$ and $\|E_B\|_2 \le c\epsilon\|B\|_2$
\cite{CampsThesis}. When one of the blocks is of size two, numerical evidence suggest
that the same strategy can be applied; a
theoretical backward error analysis is lacking, however.

If both of the blocks are of size two, the transformations are computed according to
\cite{Kagstrom1993,Kagstrom1996} relying on solving the Sylvester equation. However, this
method is not norm-wise backward stable. It leads to occasional non-negligible
off-diagonal blocks when an ill-conditioned $2{\times}2$ block close to convergence is
involved. In this case, a step of iterative refinement takes place \cite{Camps2019} to decrease the
norm of the off-diagonal block. Our numerical experiments indicated that iterative
refinement is required in about $5\%$ of all $2{\times}2$ with $2{\times}2$ swaps during a
typical RQZ iteration. In \Cref{sec:numerics} the numerics illustrate that these swaps lead
to a backward stable algorithm.

\subsection{Blocking and packing poles}

To avoid shift blurring, we avoid  large  multiplicities. However, only chasing a few
 at a time does not perform well on modern computer architectures.  Braman, Byers,
and Mathias \cite{Braman2002} proposed using a train of shifts.  Instead of introducing a shift and chasing it all the way to the bottom of the
pencil, a shift is chased just far enough to make room for another shift to be
introduced. As a result, we end up with a train of shifts that can be chased to the bottom
efficiently. The same idea was applied to the QZ algorithm by K{\aa}gstr{\"o}m and
Kressner \cite{Kagstrom2007}. 

We also make use of blocking, which tries to make optimal use of the computer architecture by enhancing
cache efficiency as a result of reducing the cost of moving data in and out of memory. The
direct updates are therefore limited to a moving computational window
(a block on the diagonal of the pencil) in which the train of shifts is chased from the
top-left corner of the window to the bottom-right. During this chase all executed
transformations are accumulated and the rest of the pencil is updated at once via
matrix-matrix multiplications.  To execute these
matrix-matrix multiplications we rely on highly optimized BLAS implementations
\cite{blackford2002updated,wang2014intel}.

The computational efficiency will benefit from getting as many shifts as possible in a
single window. Camps et al.\ \cite{CaVaWaMa20} illustrated that the
rational QZ method allows to pack shift-poles
optimally in a straightforward way. In the bulge chasing setting one can easily pack shifts tightly \cite{Braman2002}, but more
advanced techniques are required to pack them optimally \cite{KaKrLa14}.

The sliding window size $n_w=n_s+k$, where $n_s$ stands for the number of optimally
packed shift(-pole)s that are present in a window and $k$ denotes the number of positions that
each shift-pole can move down.
To select the window size Karlsson et al.\ \cite{KaKrLa14} chose to minimize
the amount of flops during a full sweep. This selection of the window size is optimal
in a sense, but it neglects other factors that can be important in practice. While the
computations inside the window are usually negligable when purely counting flops, they are in essence computationally more demanding
than the highly optimized matrix-matrix multiplications (BLAS) to update the
off-diagonal part of the pencil. As a result the calculated optimum will be an
overestimate  when using a
large number of shifts. Secondly, most BLAS implementations, and especially parallel
variants of BLAS, become more efficient when doing larger multiplications. This indicates
that the calculated optimum will be an underestimate when using a small number of shifts.

To account for the overestimation we propose to select a blocksize $n_w= n_s+ k$ that minimizes the function
\begin{equation}\label{eq:blocksizecost}
\text{cost}(k) = \frac{2cn(k+n_s)^2 + 4n_s k(k+n_s)}{k},
\end{equation}
where $c$ is the inverse of the flop rate of the BLAS calls relative to the window update. This minimum is attained for
\begin{equation}
k = n_s (1 + \frac{2n_s}{cn})^{-\frac{1}{2}}.
\end{equation}
For very large values of $c$ or $n$, this will tend towards $n_s$, the optimum obtained
in \cite{KaKrLa14}. This value will still be an underestimate for small
values of $n_s$. Therefore we impose a minimal blocksize based on the size of the
pencil. In our implementation, $c$ is set to $0.1$.

\subsection{Aggressive deflation windows and number of shifts}

The parameters $n_s$, which stands for the number of shifts, and $w_e$ and $w_s$, which are
the windows for aggressive deflation at the top and at the bottom, are harder to
select. These parameters have a significant impact on both convergence speed and the
overall execution time.
We have selected the parameters empirically. A summary is found in
Table~\ref{tab:settings}. The first column lists the size of the pencil.  The second
column lists the batch size $n_s$ of shifts that are handled in one iteration.  The third
column lists the window size $w_e$ for aggressive early deflation at the bottom-right side
of the pencil. Finally, the fourth column lists the window size $w_s$ for aggressive early
deflation at the upper-left side of the pencil. If, in the AED window $8\%$, w.r.t.\ $w_e$
or $w_s$, of the
eigenvalues are found, we do another step of AED and keep repeating this procedure before
starting the RQZ sweep.

\begin{table}[htp]
	\centering
	\caption{Settings: $n$ problem size, $m$ step multiplicity, $w_e$ AED window size at the bottom-right side of the pencil, $w_s$
		AED window size at the upper-left side of the pencil.}
	\begin{tabular}{l|c|c|c|c}
		$n$ & $n_s$ & $w_e$ & $w_s$\\
		\hline
		$\left[1;80\right[$ & $1$---$2$ & $1$---$2$ &$1$---$2$\\
		$\left[80;150\right[$ & $4$ & $8$ &$4$\\
		$\left[150;590\right[$ & $32$  & $48$ &$32$\\
		$\left[590;3000\right[$ & $40$  & $96$ &$40$\\
		$\left[3000;\infty\right[$ & $64$  & $96$ &$64$\\
	\end{tabular}
	\label{tab:settings}
\end{table}

Due to the nature of the algorithm the poles move slower to the top than shift-poles to
the bottom, hence $w_s \leq w_e$. Alternative versions of swapping
   algorithms where we get equally fast convergence to the top are proposed by Camps et
   al. \cite{CaVaWaMa20} and use bidirectional chasing.






\subsection{Full algorithm}
A  double-shift real rational QZ step (DRRQZ) step with all add-ons proceeds as follows. The heuristics were discussed
in the previous sections.

\begin{enumerate}[I.]
	
\item Check for interior deflations to select the active part of the matrix. 
		\item AED at the top.
          \begin{enumerate}
          \item Check for deflations at the upper-left side of the pencil using AED. If
            sufficient eigenvalues are found redo the AED.
          \item Take $n_s$ undeflated eigenvalues to be
            introduced as poles at the end of the pencil in III(c).
          \end{enumerate}
	\item AED at the bottom.
          \begin{enumerate}
          \item Check for deflations at the bottom-right side of the pencil using AED, and
            repeat if necessary.
          \item Store $n_s$ undeflated eigenvalues to be introduced as shift-poles.

          \item The undeflated eigenvalues from II(b) are introduced as poles at the
            bottom and moved to the top of the window\footnote{They must be moved to
              the top of the window, otherwise the next AED step removes them               again.}.
          \end{enumerate}
	
	\item Initiate the sweep by introducing $n_s$ shift-poles from step III(b).
	The involved transformations are accumulated and the pencil is updated by
	level~3 BLAS matrix-matrix multiplication.
	
	\item Chase the batch of $n_s$ shift-poles to the last $n_s$ positions on the subdiagonal
	of the block Hessenberg pencil. The chasing is performed by repeatedly swapping
	the $n_s$ shift-poles with the next $k$ poles thereby using blocking. 
	
\end{enumerate}


\section{Numerical experiments}
\label{sec:numerics}

In this section we will demonstrate the accuracy and efficiency of our implementation,
which is a double-shift and double-pole version for handling real pencils, and which sticks to real
arithmetic. We name our algorithm \texttt{DRRQZ}. We will show that it is much faster than our
complex single shift rational QZ algorithm, named \texttt{ZRQZ}, and we will compare with the QZ
algorithm of \texttt{LAPACK} and \texttt{PDHGEQZ} \cite{adlerborn2015pdhgeqz}. \texttt{PDHGEQZ} is a library designed to be used on large parallel machines on very large problems. For a fair comparison with our code, \texttt{PDHGEQZ} was not run in parallel, but linked to the same parallel BLAS implementation. Even run serially, \texttt{PDHGEQZ} performs very well \cite{adlerborn2014parallel}.

The QZ algorithm based on bulge chasing and the \texttt{DRRQZ} implementation have some essential
differences. 
\begin{itemize}
\item In the QZ case we always swap (implemented as a bulge chase) the shift-pole block
  which is of size $2\times 2$ with a $1\times 1$ pole block, since the poles moving up
  always equal $\infty$. In the \texttt{DRRQZ} we can also encounter $2 \times 2$ pole blocks to be
  swapped with the shift-pole block, so some additional branches in the code are required. 
\item The swapping procedure is implemented completely differently. The QZ setting allows
  us to
  execute all operations on one side of the pencil, before handling the other side, in the
  \texttt{DRRQZ} setting this is not true, so we pay a price in cache efficiency. 
\end{itemize}
 There are also some particularities of the codes we will compare with. 
  \begin{itemize}
    \item \texttt{PDHGEQZ} has a special built in feature to deflate infinite eigenvalues, that
      appear as zeroes on the diagonal of $B$, while executing the algorithm; this feature
      is unexisting in  the rational QZ algorithm and appears to be hard to implement
      generic Hessenberg pencils.
    \item  The \texttt{DRRQZ} setting allows us to trivially deal with optimally packed
      shift-poles \cite{CaVaWaMa20}. They are more tightly packed than the ones in \texttt{PDHGEQZ} and \texttt{LAPACK}.
\item \texttt{LAPACK} doesn't use AED, \texttt{PHDGEQZ} and \texttt{DRRQZ} do, \texttt{DRRQZ} also executes AED on the top of the pencil
\end{itemize}

The numerical tests have been performed on an
Intel Xeon E5-2697 v3 CPU with $14$ cores and $128$GB of RAM.
Our \texttt{DRRQZ} implementation
with aggressive deflation is compiled with \emph{gfortran} version 7.4.0 using compilation
flag \texttt{-Ofast}. The code is linked with release 88 of \texttt{MKL} and the link to the git
repository is available at \url{numa.cs.kuleuven.be/software/rqz} and
\url{github.com/thijssteel/multishift-multipole-rqz}.

\subsection{Simple test problems}
\label{ssec:kressnerlapack}

For our first numerical experiment, we will compare the implementations using two different pencils.
The first pencil is in Hessenberg, triangular form, with the nonzero entries drawn from a
uniform distribution between 0 and 1. This is the Hessrand2 test problem from \cite{adlerborn2014parallel}.
Because of the random nature of this pencils, the results are
averaged over 10 runs\footnote{Except for \texttt{LAPACK}, there we considered a single run due to high
  computational cost.}.
AED is so effective on this pencil that barely any sweeps are required.
Most of the computation time is therefore spent in AED.
The second pencil is also in Hessenberg, triangular form, with the nonzero entries given
by $A_{i,j} = i+j$ and $B_{i,j} = 2i+3j$. This pencil was constructed explicitly such that 
AED is less effective. As a consequence we  expect the sweep to have a bigger role in the execution time.

\begin{figure}[htp]
\label{fig:timereal}
\centering
\resizebox{\textwidth}{!}{%
\begin{tikzpicture}

\pgfplotstableread[row sep=\\,col sep=&]{
  sz & tqz & tqzm & tqzkressner & trqzm \\
  1000 & 7.1884552 &  1.085662682 &  1.471481982 & 1.0 \\
  1414 & 25.9454376 &  1.570285445 & 2.369853147 & 1.200966664\\
  2000 & 93.8453857 &  2.423607355 & 4.109919943 & 1.881835564\\
  2828 & 276.3524980 &  3.738139845 & 7.753459501 & 2.824150818\\
  4000 & 771.3608919 & 6.013744927 & 15.33716692 & 5.453504109\\
  5657 & 2226.5918753 & 9.979186127 & 27.06582937 & 9.445276791\\
}\mydata

\begin{axis}[
    xticklabels={1000,1414,2000,2828,4000,5657,8000},
    ymode=log,
    xmode=log,
    xtick=data,
    bar width=2mm,
    x tick label style={rotate=45,anchor=east},
    xlabel={$n$},
    ylabel={$t(s)$},
    my ybar legend,
    legend pos=north west,
    legend style={font=\small},
    width=0.50\linewidth,
    at={(0\linewidth,0)},
]
	\addplot[ybar,bar shift=-3mm,fill=black!10,postaction={pattern=dots},opacity=0.6,legend] table[x=sz,y=trqzm]{\mydata};
	\addplot[ybar,bar shift=-1mm,fill=black!50,postaction={pattern=north east lines},opacity=0.6,legend] table[x=sz,y=tqzkressner]{\mydata};
	\addplot[ybar,bar shift=1mm,fill=black!50,opacity=0.6,legend] table[x=sz,y=tqz]{\mydata};
%
%

	 \legend{\texttt{DRRQZ},\texttt{PDHGEQZ},\texttt{LAPACK}}
\end{axis}

\pgfplotstableread[row sep=\\,col sep=&]{
	sz & tqz & tqzm & tqzkressner & trqzm \\
	1000 & 9.1470463 &  1.9312363 &  4.0088403300615028 & 2.3975728\\
	1414 & 30.6765495 &  3.1781807 & 5.8239643430570140 & 3.9779007\\
	2000 & 103.8768633 &  5.7166858 & 12.459904983988963 & 6.9142253\\
	2828 & 342.5666382 &  11.3852339 & 30.210247943992727 & 14.0006025\\
	4000 & 925.2474879 &  24.5866870 & 57.969269493012689 & 32.3396838\\
	5657 & 2898.8343897 & 56.1577301 & 123.86925438593607 & 66.2084222\\
	8000 & 8192.5145759 & 146.4358318 & 438.78187234397046 & 167.0603150\\
}\mydata

\begin{axis}[
xticklabels={1000,1414,2000,2828,4000,5657,8000},
ymode=log,
xmode=log,
xtick=data,
bar width=2mm,
x tick label style={rotate=45,anchor=east},
xlabel={$n$},
my ybar legend,
legend pos=north west,
legend style={font=\small},
width=0.50\linewidth,
at={(0.5\linewidth,0)},
]
\addplot[ybar,bar shift=-3mm,fill=black!10,postaction={pattern=dots},opacity=0.6,legend] table[x=sz,y=trqzm]{\mydata};
\addplot[ybar,bar shift=-1mm,fill=black!50,postaction={pattern=north east lines},opacity=0.6,legend] table[x=sz,y=tqzkressner]{\mydata};
\addplot[ybar,bar shift=1mm,fill=black!50,opacity=0.6,legend] table[x=sz,y=tqz]{\mydata};

%
%

\legend{\texttt{DRRQZ},\texttt{PDHGEQZ},\texttt{LAPACK}}
\end{axis}

\end{tikzpicture}
}
\caption{Execution time of \texttt{DHGEQZ} of \texttt{LAPACK}, \texttt{libRQZ} and \texttt{PDHGEQZ} on randomly generated real-valued matrix pencils (\emph{left}) and the `$i+j$' pencil (\emph{right}).}
\end{figure}
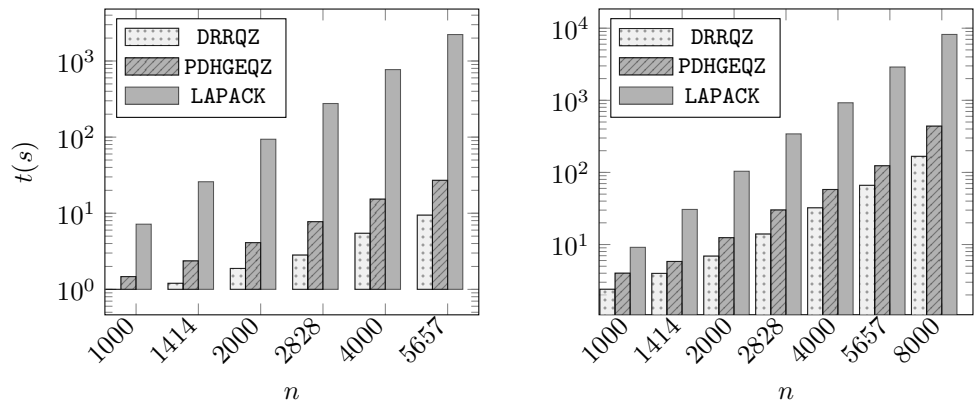

\Cref{fig:timereal} shows the execution time of \texttt{DRRQZ},
\texttt{DHGEQZ} and \texttt{PDHGEQZ} for problems of size $1000$ up to $8000$ on a loglog scale.
The left part shows the results for the randomly generated pencil. \texttt{libRQZ} is
faster than \texttt{PDHGEQZ} and both methods show large speedups over \texttt{LAPACK}
(which does not use AED). The right part shows the results for the second `$i+j$'
pencil. The larger amount of sweeps results in a smaller speedup over both competitors.

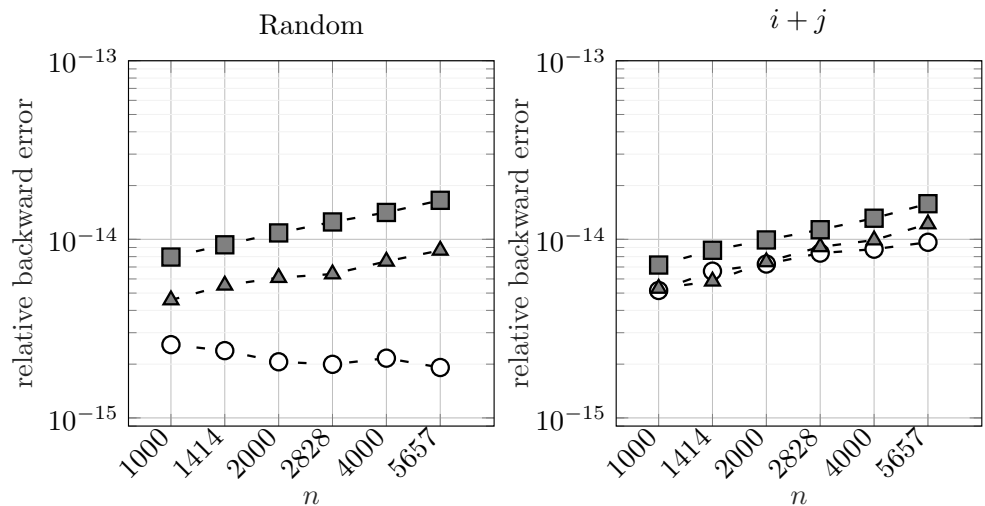
\begin{figure}[htp]
	\label{fig:bwerealcomplex}
	\centering
	\resizebox{\textwidth}{!}{%
		\begin{tikzpicture}

\pgfplotstableread[row sep=\\,col sep=&]{
	n & eRQZ & eKressner & eLAPACK \\
	1000 & 2.5714143E-15 &  4.5714045984879848E-015 & 7.9596449E-15 \\
	1414 & 2.3806355E-15 &  5.5271398026622763E-015 & 9.3180672E-15  \\
	2000 & 2.0619174E-15 &  6.0984299572949447E-015 & 1.0860809E-14\\
	2828 & 1.9956699E-15 &  6.3906114727257271E-015 & 1.2511191E-14\\
	4000 & 2.1599589E-15 &  7.5249656379195053E-015 & 1.4141155E-14  \\
	5657 & 1.9153082E-15 &  8.6699884513419090E-015 & 1.6539671E-14\\
}\myrandomdata

\pgfplotstableread[row sep=\\,col sep=&]{
	n & eRQZ & eKressner & eLAPACK \\
	1000 & 5.1709481E-15 &  5.3066257138206025E-015 & 7.1791612E-15 \\
	1414 & 6.6354981E-15 &  5.8231295483828774E-015 & 8.6841928E-15  \\
	2000 & 7.2721973E-15 &  7.4712305743333309E-015 & 9.9210641E-15 \\
	2828 & 8.3719760E-15 &  9.0686153704971901E-015 & 1.1318501E-14\\
	4000 & 8.8067541E-15 &  9.8994357970937368E-015 & 1.3126660E-14  \\
	5657 & 9.6262520E-15 &  1.2182817664622218E-014 & 1.5830244E-14 \\
}\myijdata

\pgfplotsset{
  log x ticks with fixed point/.style={
      xticklabel={
        \pgfkeys{/pgf/fpu=true}
        \pgfmathparse{exp(\tick)}%
        \pgfmathprintnumber[fixed relative, precision=3]{\pgfmathresult}
        \pgfkeys{/pgf/fpu=false}
      }
  }}
  
\begin{axis}[%
at={(0cm,0cm)},
width=4.5cm,
height=4.5cm,
title={Random},
ylabel shift = -0.05 cm,
xlabel shift = -0.15 cm,
scale only axis,
xmode=log,
xmin=761,
xmax=8000,
log x ticks with fixed point,
xtick=data,
xticklabels={1000,1414,2000,2828,4000,5657},
xlabel style={font=\color{white!15!black}},
xlabel={$n$},
x tick label style={rotate=45,anchor=east},
ymode=log,
ymin=0.9e-15,
ymax=1e-13,
yminorticks=true,
ylabel style={font=\color{white!15!black}},
ylabel={relative backward error},
axis background/.style={fill=white},
every axis plot/.append style={thick},
grid=both,
grid style={line width=.1pt, draw=gray!10},
major grid style={line width=.2pt,draw=gray!50},
]

\addplot [color=black, loosely dashed, mark=*, mark options={solid, scale=1.5, fill=white}]
  table[x=n,y=eRQZ]{\myrandomdata};

\addplot [color=black, loosely dashed, mark=triangle*, mark options={solid, scale=1.5, fill=gray}]
  table[x=n,y=eKressner]{\myrandomdata};

\addplot [color=black, loosely dashed, mark=square*, mark options={solid, scale=1.5, fill=gray}]
  table[x=n,y=eLAPACK]{\myrandomdata};

\end{axis}

\begin{axis}[%
at={(6cm,0cm)},
width=4.5cm,
height=4.5cm,
title={$i+j$},
ylabel shift = -0.15cm,
xlabel shift = -0.15cm,
scale only axis,
xmode=log,
xmin=761,
xmax=8000,
xminorticks=true,
log x ticks with fixed point,
xtick=data,
xticklabels={1000,1414,2000,2828,4000,5657},
xlabel style={font=\color{white!15!black}},
x tick label style={rotate=45,anchor=east},
xlabel={$n$},
ymode=log,
ymin=0.9e-15,
ymax=1e-13,
ylabel style={font=\color{white!15!black}},
ylabel={relative backward error},
axis background/.style={fill=white},
every axis plot/.append style={thick},
grid=both,
grid style={line width=.1pt, draw=gray!10},
major grid style={line width=.2pt,draw=gray!50},
]

\addplot [color=black, loosely dashed, mark=*, mark options={solid, scale=1.5, fill=white}]
  table[x=n,y=eRQZ]{\myijdata};

\addplot [color=black, loosely dashed, mark=triangle*, mark options={solid, scale=1.5, fill=gray}]
  table[x=n,y=eKressner]{\myijdata};

\addplot [color=black, loosely dashed, mark=square*, mark options={solid, scale=1.5, fill=gray}]
  table[x=n,y=eLAPACK]{\myijdata};

\end{axis}

\end{tikzpicture}%
	}
	\caption{Maximum of the relative backward error on $A$ and $B$ of Schur decomposition computed with \texttt{LAPACK} (\emph{squares}), \texttt{PDHGEQZ} (\emph{triangles})
		and \texttt{DRRQZ} from \texttt{libRQZ} (\emph{circles}). Both
		for randomly generated pencils (\emph{left}) and the `$i+j$' pencil (\emph{right}).}
\end{figure}

\Cref{fig:bwerealcomplex} shows the relative backward errors,
$$
\|S - Q^T A Z\|_{F} / \|A\|_{F},
\quad \text{and,} \quad
\|T - Q^T B Z\|_{F} / \|B\|_{F},
$$
on the generalized real Schur decompositions obtained with \texttt{libRQZ}, \texttt{PDHGEQZ} and \texttt{DHGEQZ}.
We observe that the relative backward errors of \texttt{libRQZ} are almost always the smallest.

\subsection{Complex code and real double-pole code}
We compare the \texttt{DRRQZ} code with the RQZ code, the latter will just ignore the fact that the
matrix is real, and use complex shift-poles to find the complex conjugate pairs of
eigenvalues. Complex arithmetic is between 2 and 6 times more expensive than real
arithmetic. However, the swapping procedure is significantly simpler because all the
pole blocks are of size 1 so we do not expect the real code to achieve speedups as large
as 6. The numerical experiments reveal that around $5000$, \texttt{DRRQZ} becomes twice as fast.

We solved the `$i+j$' pencil  using the real and complex RQZ code. The results are
displayed in \Cref{tab:numexp_realvscomplex}. The real code is always faster with speedups
of up to $2.4$ for the larger pencils.

\begin{table}[htp]
    \centering
    \caption{Execution time and relative backward error of real and complex RQZ code on the `$i+j$` pencil.}
    \begin{adjustbox}{max width=\textwidth}
        \begin{tabular}{c|c|c|c|c|c}
            \multicolumn{1}{c|}{}  & \multicolumn{2}{c|}{\texttt{RQZ}} &
            \multicolumn{2}{c|}{\texttt{DRRQZ}} & \multicolumn{1}{c}{}\\
          $n$ & $t(s)$ & max$(\Delta A,\Delta B)$ & $t(s)$ & max$(\Delta A,\Delta B)$ & Speedup\\
            \hline
            $1000$ & $2.56$ & $5.48 \cdot 10^{-15}$ & $2.12$ & $5.17 \cdot 10^{-15}$ & 1.21\\
            $1414$ & $4.96$ & $5.93 \cdot 10^{-15}$ & $3.64$ & $6.64 \cdot 10^{-15}$ & 1.36\\
            $2000$ & $10.68$ & $6.35 \cdot 10^{-15}$ & $7.03$ & $7.27 \cdot 10^{-15}$ & 1.52\\
            $2828$ & $24.05$ & $6.9 \cdot 10^{-15}$ & $14.19$ & $8.37 \cdot 10^{-15}$ & 1.69\\
            $4000$ & $57.59$ & $7.7 \cdot 10^{-15}$ & $31.24$ & $8.57 \cdot 10^{-15}$ & 1.84\\
            $5657$ & $144.27$ & $8.2 \cdot 10^{-15}$ & $65.17$ & $9.62 \cdot 10^{-15}$ & 2.21\\
            $8000$ & $375.75$ & $9.1 \cdot 10^{-15}$ & $157.11$ & $9.89 \cdot 10^{-15}$ & 2.39\\
        \end{tabular}
    \end{adjustbox}
    \label{tab:numexp_realvscomplex}
\end{table}



\subsection{Problems from applications}

In this section we test \texttt{libRQZ} on five pencils originating from
applications.
We study the \emph{cavity} and \emph{obstacle flow} pencils generated
with IFISS \cite{Elman07,Elman14}. The same pencils were studied in
our initial paper on the RQZ method \cite{camps2019rational}.
Besides these pencils, we have selected two pencils from
Matrix market \cite{Boisvert1997} originating from the MHD collection
and the \emph{rail} pencil from the Oberwolfach benchmark collection
\cite{Oberwolfach}.

Before being able to run the eigenvalue solvers. The pencils need to be reduced to
Hessenberg, triangular form. At this moment there is no competitive direct method
to reduce a pencil to Hessenberg, Hessenberg form, with a prescribed set of poles. The lack of such a reduction is an
important aspect of future research.  We do note, however, that the RQZ method accepts a
more general input form, and for instance pencils resulting from rational Krylov methods
\cite{BeGu15} would no longer need a reduction step.  

Our code uses a basic input format, whereas \texttt{PDHGEQZ} makes use of
\texttt{ScaLAPACK}'s data format. As a consequence our current implementation can not make direct
use of the superior built-in reduction to Hessenberg, triangular form in
\texttt{PDHGEQZ}. Instead we have to rely on \texttt{DGEQRFP} and \texttt{DGGHD3} of
\texttt{LAPACK}.


The results of the numerical tests are summarized in \Cref{tab:numexp}.  The table lists
the execution time of both the reduction and iterative steps and the maximum of the
relative backward errors of $A$ and $B$ for the generalized real Schur form.
\texttt{libRQZ} is better in terms of execution time of the iterative part and the
backward error.  We stress, that for both methods, the reduction to Hessenberg, triangular
form requires much more computation time than the iterative part. Work and research is required to
either adapt our code to operate on the \texttt{ScaLAPACK} data format, or to devise a fast
reduction to Hessenberg, Hessenberg form.

\begin{table}[htp]
    \centering
    \caption{Execution times of the reduction ($t_{\text{HT}}(s)$) and iterative ($t_{\text{IT}}(s)$) step and maximum relative backward error on the generalized quasi Schur form computed with \texttt{DRRQZ} and \texttt{PDHGEQZ}
    for pencils originating from applications.}
    \begin{adjustbox}{max width=\textwidth}
    \begin{tabular}{l|c|c|c|c|c|c|c}
    \multicolumn{2}{c|}{}  & \multicolumn{3}{c|}{\texttt{DRRQZ}} & \multicolumn{3}{c}{\texttt{PDHGEQZ}} \\
    Problem & $n$ & $t_{\text{HT}}(s)$ & $t_{\text{IT}}(s)$ & max$(\Delta A,\Delta B)$ & $t_{\text{HT}}(s)$ & $t_{\text{IT}}(s)$ & max$(\Delta A,\Delta B)$ \\
    \hline
    Obstacle Flow & $2488$ & $104.95$ & $5.3$ & $8.2 \cdot 10^{-15}$ & $24.4$ & $7.3$ & $1.0 \cdot 10^{-14}$\\
    Cavity Flow & $2467$ & $105.4$ & $4.4$ & $6.3 \cdot 10^{-15}$ & $26.5$ & $6.3$ & $9.0 \cdot 10^{-15}$\\
    MHD3200 & $3200$ & $230.1$ & $5.0$ & $2.9 \cdot 10^{-15}$ & $54.6$ & $9.1$ & $9.2 \cdot 10^{-15}$\\
    MHD4800 & $4800$ & $915.1$ & $15.2$ & $4.0 \cdot 10^{-15}$ & $184.3$ & $22.7$ & $1.5 \cdot 10^{-14}$\\
    RAIL & $5177$ & $1215.1$ & $28.5$ & $1.1 \cdot 10^{-15}$ & $258.7$ & $47.3$ & $3.5 \cdot 10^{-14}$
    \end{tabular}
    \end{adjustbox}
    \label{tab:numexp}
\end{table}





\section{Theoretical analysis of the rational QZ algorithm}
\label{sec:theory}
In this section we will prove that the rational QZ algorithm is equivalent to subspace iteration driven
by a rational function. We will rely heavily on the connection with rational Krylov. The
proof and analysis proceeds similarly to the work of Watkins \cite{b333}. A more detailed
analysis focussing of the single shift rational QZ case  can be found in Camps et al.\ \cite{camps2019rational} and the references therein.

\subsection{Rational Krylov and block Hessenberg pencils}
\label{ssec:RK}

We study the structure of \emph{rational Krylov subspaces}
generated by proper block Hessenberg pencils.
These results prove the correctness of the \emph{pole introduction}
operation (see \Cref{ssec:manipulating}) and
\emph{essential uniqueness} 
of a multishift, multipole RQZ step.

The elementary rational matrices $M$ and $N$ \eqref{eq:elemRational} are used to construct
rational Krylov matrices generated by a regular matrix pencil, a starting vector,
and a tuple of poles.
They  satisfy some basic properties we need later on, for proofs see \cite[Lemma 5.3]{camps2019rational}).
The inverse $M(\varrho,\xi)^{-1}$ is defined if $\varrho \notin \Lambda$ and is
equal to $M(\xi,\varrho)$.
They are commutative,
$M(\varrho_1,\xi_1) M(\varrho_2,\xi_2) = M(\varrho_2,\xi_2) M(\varrho_1,\xi_1)$, and
they can be merged together,
$M(\varrho,\xi_1) M(\xi_1,\xi_2) = M(\varrho,\xi_2)$,
if a pole and shift are equal.
Analogous results hold for $N(\varrho,\xi)$.


\begin{definition}[Rational Krylov matrices]
\label{def:ratkrylmat}
Let $A,B \in \FF^{n{\times}n}$ form a regular matrix pencil,
$\bm{v} \in \FF^n$ a nonzero vector, $k{\leq}n$,
$\Xi = (\xi_1, \hdots, \xi_{k-1})$
a tuple of poles distinct from the eigenvalues, and
$P = (\varrho_1, \hdots, \varrho_{k-1}) \subset \bar{\CC}$ 
a tuple of shifts distinct from the poles.
The corresponding rational Krylov matrices are defined as: 
\begin{equation}
\label{eq:ratkrylmat}
\resizebox{.89\hsize}{!}{$
\begin{split}
K^{\text{rat}}_{k}(A,B,\bm{v}, \Xi, P) & = 
\left[ \bm{v}, 
M(\varrho_1,\xi_1) \bm{v}, 
M(\varrho_2,\xi_2) M(\varrho_1,\xi_1) \bm{v},
\, \hdots, 
\left( \prod_{i=1}^{k-1}M(\varrho_i,\xi_i) \right) \bm{v} \right], \\
L^{\text{rat}}_{k}(A,B,\bm{v}, \Xi, P) & = 
\left[ \bm{v}, 
N(\varrho_1,\xi_1)  \bm{v}, 
N(\varrho_2,\xi_2) N(\varrho_1,\xi_1) \bm{v},
\, \hdots, 
\left( \prod_{i=1}^{k-1} N(\varrho_i,\xi_i) \right) \bm{v} \right].
\end{split}
$}
\end{equation}
\end{definition}
The column spaces of these matrices span the
\emph{rational Krylov subspaces}.

\begin{definition}[Rational Krylov subspaces]
\label{def:rksubspace}
The rational Krylov subspaces $\mathcal{K}_{k}^{\text{rat}}$ and
$\mathcal{L}_{k}^{\text{rat}}$, $k{\leq}n$,
associated with the ${n{\times}n}$ regular pencil $(A,B)$, 
a nonzero vector $\bm{v} \in \FF^n$, and pole tuple
$\Xi = (\xi_1, \hdots, \xi_{k-1})$ distinct from the
eigenvalues, are defined as,
\begin{equation}
\label{eq:rksubspace_definition}
\begin{split}
\mathcal{K}_{k}^{\text{rat}}(A,B,\bm{v},\Xi) 
& \equiv \mathcal{R}(K_{k}^{\text{rat}}(A,B,\bm{v}, \Xi, \Rho))
= \prod_{i=1}^{k-1} M(\hat{\varrho},\xi_i) \cdot
\mathcal{K}_k(M(\check{\varrho},\hat{\varrho}),\bm{v})
, \\
\mathcal{L}_{k}^{\text{rat}}(A,B,\bm{v},\Xi) 
& \equiv \mathcal{R}(L_{k}^{\text{rat}}(A,B,\bm{v}, \Xi, \Rho))
= \prod_{i=1}^{k-1} N(\hat{\varrho},\xi_i) \cdot
\mathcal{K}_k(N(\check{\varrho},\hat{\varrho}),\bm{v})
,
\end{split}
\end{equation}
where the shift tuple $\Rho$ is freely chosen
in agreement with \Cref{def:ratkrylmat},
$\hat{\varrho}$ is a shift different from the eigenvalues and poles, 
and $\check{\varrho}$ is an alternative shift different from $\hat{\varrho}$.
\end{definition}

The first equality in \cref{eq:rksubspace_definition} defines
the rational Krylov subspaces, the second equality
repeats \cite[Lemma 5.6.II]{camps2019rational}.
This result shows that rational Krylov subspaces are \emph{shift invariant}
as they are independent of the choice of shifts $\Rho$.

The following theorem generalizes \cite[Theorem 5.6]{camps2019rational}
and shows that the rational Krylov subspaces $\mathcal{K}^{\text{rat}}$ and
$\mathcal{L}^{\text{rat}}$ stemming from proper block Hessenberg pencils link closely to subspaces built from the standard basis
vectors. This theorem is essential in proving the link with subspace iteration in \Cref{sec:convergence}.
\begin{theorem}
\label{thm:blockHessspaces}
Given an $n{\times}n$ proper block Hessenberg pencil $(A,B)$ with partition 
$\bm{s} = (s_1$, $\hdots$, $s_m)$,
poles $\Xi = (\Xi^1, \hdots, \Xi^m)$ with 
$\Xi^i = \lbrace \xi^{i}_{1}, \hdots \xi^{i}_{s_i} \rbrace$ that are all different from
the eigenvalues.
Then for $j=0,1,\hdots,m$,
\begin{equation}
\label{eq:Kratspan}
\mathcal{K}^{\text{rat}}_{s_1+\cdots s_{j}+1}
(A,B,\bm{e}_1, (\Xi^1, \hdots, \Xi^j)) 
= \mathcal{E}_{s_1+\cdots+s_{j}+1}.
\end{equation}
While for $j=1,\hdots,m$,
\begin{equation}
\label{eq:Lratspan}
\mathcal{L}^{\text{rat}}_{s_1+\cdots+s_{j}}
(A,B,\bm{z}_1, (\breve{\Xi}^1, \Xi^2, \hdots, \Xi^j)) 
= \mathcal{E}_{s_1+\cdots+s_{j}},
\end{equation}
with $\breve{\Xi}^1 = \lbrace \xi^1_1, \hdots, \xi^1_{s_{1}-1} \rbrace$,
and $\bm{z}_1$ the right eigenvector of the pole pencil corresponding to
pole $\xi^1_{s_1}$. Here $\xi^1_{s_1}$ can be any of the poles in $\Xi^1$.
\end{theorem}

\begin{proof}
  We rely on the transformation $(\hat{A}, \hat{B}) = Q^* (A,B) Z$ from proper block
  Hessenberg pencil $(A,B)$ to proper Hessenberg pencil $(\hat{A},\hat{B})$ as proposed in
  \Cref{thm:implicitQ}, relying on block diagonal unitary matrices. We know that Theorem
  5.6 of Camps et al.\ \cite{camps2019rational} links rational Krylov subspaces built from
  Hessenberg pencils to the subspaces $\mathcal{E}_j$; the block diagonal unitary
  equivalence transformation messes this up a little, the equality is now only valid for
  particular values of $j$, such as proposed in \eqref{eq:Kratspan}.

\Cref{eq:Lratspan} is slightly more involved. The transformation to Hessenberg pencil
requires to use a new elementary rational matrix 
$\hat{N}(\varrho,\xi) = Z^* N(\varrho,\xi) Z$. With this $\hat{N}$ we can fall back on
\cite[Theorem 5.6]{camps2019rational} to prove the theorem.
%
\end{proof}

Again we see that operating with block Hessenberg pencils destroys the properties in the
blocks, but in between blocks theorems as in \cite{camps2019rational} remain true.

\subsection{Correctness of the algorithm}
We prove correctness of the procedure (see \Cref{ssec:manipulating}) to introduce the poles in the
algorithm, that is, replacing the first $\ell$ poles by $\ell$ new ones, named the shift-poles.
From \cref{eq:rksubspace_definition} and \Cref{thm:blockHessspaces} we have 
that,
\begin{equation}
  \label{eq:vecx}
\bm{x} \in \mathcal{K}^{\text{rat}}_{\ell+1}(A,B,\bm{e}_1,\Xi) = \mathcal{E}_{\ell+1}.
\end{equation}
This implies that $Q$ in \cref{eq:Qmultishift} is of the form $\text{diag}(Q_{\ell+1},I)$,
with $Q_{\ell+1}$ an $(\ell{+}1){\times}(\ell{+}1)$ unitary matrix.
Furthermore, for $j = 0, 1, \hdots, m-i+1$,
\begin{eqnarray*}
& & \mathcal{K}^{\text{rat}}_{\hat{s}_1+\cdots+\hat{s}_{j}+1}(\hat{A},\hat{B},\bm{e}_1,
(\Rho, \Xi^{i+1}, \hdots, \Xi^m)) \\
& = & \prod_{k=1}^{\hat{s}_1+\cdots+\hat{s}_{j}} \hat{M}(\hat{\varrho},\hat{\xi}_k) \cdot \mathcal{K}_{\hat{s}_1+\cdots+\hat{s}_{j}+1}(\hat{M}(\check{\varrho},\hat{\varrho}),\bm{e}_1) \\
& = &Q^* M(\hat{\varrho},\varrho_1) \cdots M(\hat{\varrho},\varrho_{\ell}) M(\hat{\varrho},\xi_{\ell+1}) \cdots M(\hat{\varrho},\xi_{\hat{s}_1+\cdots+\hat{s}_{j}}) \cdot \mathcal{K}_{\hat{s}_1+\cdots+\hat{s}_{j}+1}(M(\check{\varrho},\hat{\varrho}),\bm{q}_1) \\
& = &Q^* M(\hat{\varrho},\varrho_1) \cdots M(\hat{\varrho},\varrho_{\ell})
      M(\hat{\varrho},\xi_{\ell+1}) \cdots
      M(\hat{\varrho},\xi_{\hat{s}_1+\cdots+\hat{s}_{j}}) \\
      & & \quad \cdot \mathcal{K}_{\hat{s}_1+\cdots+\hat{s}_{j}+1}\left(M(\check{\varrho},\hat{\varrho}),\prod_{k=1}^{\ell} M(\varrho_k,\xi_k\right)\bm{e}_1) \\
& = &Q^* M(\hat{\varrho},\xi_1) \cdots M(\hat{\varrho},\xi_{\ell}) M(\hat{\varrho},\xi_{\ell+1}) \cdots M(\hat{\varrho},\xi_{\hat{s}_1+\cdots+\hat{s}_{j}}) \cdot \mathcal{K}_{\hat{s}_1+\cdots+\hat{s}_{j}+1}(M(\check{\varrho},\hat{\varrho}),\bm{e}_1) \\
& = &Q^* \mathcal{K}^{\text{rat}}_{\hat{s}_1+\cdots+\hat{s}_{j}+1}(A,B,\bm{e}_1, (\Xi^1, \hdots, \Xi^i, \Xi^{i+1}, \hdots, \Xi^m)) \\
& = & Q^* \mathcal{E}_{\hat{s}_1+\cdots+\hat{s}_{j}+1} = \mathcal{E}_{\hat{s}_1+\cdots+\hat{s}_{j}+1}.
      \end{eqnarray*}
In the first equality we used \cref{eq:rksubspace_definition},
we applied $\hat{M} = Q^* M Q$ in the
second equality, and combined \cref{eq:vec_multishift,eq:Qmultishift} to get
$\bm{q}_1 = \prod_{k=1}^{\ell} M(\varrho_k,\xi_k) \bm{e}_1$ in the third equality.
The fourth equality uses the commutativity of the $M$ matrices and the property that
$M(\hat{\varrho},\varrho_k) M(\varrho_k,\xi_k)$ can be merged to $M(\hat{\varrho},\xi_k)$.
This results in the rational Krylov subspace of the original pencil with the original
poles in the fifth equality and by \Cref{thm:blockHessspaces} we know that this is equal to
$\mathcal{E}_{\hat{s}_1+\cdots+\hat{s}_{j}+1}$.
Finally, since $Q$ has a block diagonal structure,
it does not affect the $\mathcal{E}_{\hat{s}_1+\cdots+\hat{s}_{j}+1}$ for
the given sizes.
It is clear that $(\hat{A}, \hat{B})$ is a proper block Hessenberg pencil with partition
$\hat{\bm{s}} = (\ell, s_{i+1}, \hdots, s_m)$ by construction.
The last poles are unchanged by the block diagonal structure of $Q$ and the first $\ell$
poles are changed to $\Rho$ which follows from the uniqueness of block Hessenberg pencils,
see \Cref{thm:implicitQ}.

\subsection{Convergence}
\label{sec:convergence}

In \cite[Theorem 6.1]{camps2019rational} it is shown that an RQZ step with shift $\varrho$
on a Hessenberg pencil with pole tuple $\Xi = (\xi_1, \hdots, \xi_{n-1})$ and new
pole $\xi_n$ performs nested subspace iteration accelerated by
\begin{equation}
q_{k}^{Q}(z) = \frac{z - \varrho}{z - \xi_k},
\quad \text{and} \quad
q_{k}^{Z}(z) = \frac{z - \varrho}{z - \xi_{k+1}},
\end{equation}
for the $k$th column vector of respectively $Q$ and $Z$. Based on \Cref{thm:implicitQ}
this can be extended to block Hessenberg pencils
where in the multishift, multipole RQZ method convergence will now only hold for
particular values of $k$, which correspond to the in-between block positions.
Again shifts that have been swapped along the subdiagonal
of the block Hessenberg pencil will lead to deflations at the end of the pencil, while
poles that have been moved to the front of the pencil lead to convergence of eigenvalues at the
beginning. This holds under the assumption that a good choice of poles and shifts is made.

\section{Conclusion and future work}
\label{sec:conclusion}

In this paper we have generalized the rational QZ method from Hessenberg to block
Hessenberg pencils.  This allows for the use of complex conjugate shifts and poles in real
arithmetic.  In the spirit of recent developments we used small shift and pole
multiplicities, packed optimally together.   We also implemented the aggressive early deflation
strategy for block Hessenberg pencils.  Numerical experiments indicated that this
combination leads to an efficient algorithm for the generalized eigenvalue problem that is
competitive with state of the art implementations.

Important directions for future research include an efficient reduction to Hessenberg,
Hessenberg form and investigating whether it is possible to streamline the swapping
procedure so that only a single, particular swap, e.g. $2\times 2$ with $1\times 1$ would be
used. This would simplify the code and create an additional speed up.

\section*{Acknowledgements}

The authors are grateful to Paul Van Dooren and Nicola Mastronardi for their help
with the iterative refinement procedure for $2{\times}2$ with $2{\times}2$
swaps \cite{Camps2019} which was essential for handling $2{\times}2$ blocks accurately.

Reviewers commented on the article asking major changes with respect to organization, theory,
and software; and we are grateful for these requests as they have significantly improved
this paper.

\bibliographystyle{siamplain}
\bibliography{longstrings,./bibliography_rqz2}

\begin{thebibliography}{10}

\bibitem{adlerborn2014parallel}
{\sc B.~Adlerborn, B.~K{\aa}gstr{\"o}m, and D.~Kressner}, {\em A parallel {QZ}
  algorithm for distributed memory {HPC} systems}, SIAM Journal on Scientific
  Computing, 36 (2014), pp.~C480--C503.

\bibitem{adlerborn2015pdhgeqz}
{\sc B.~Adlerborn, B.~K{\aa}gstr{\"o}m, and D.~Kressner}, {\em {PDHGEQZ} user
  guide}, 2015.

\bibitem{BeGu15}
{\sc M.~Berljafa and S.~{G\"uttel}}, {\em Generalized rational {Krylov}
  decompositions with an application to rational approximation}, SIAM Journal
  on Matrix Analysis and Applications, 36 (2015), pp.~894--916.

\bibitem{blackford2002updated}
{\sc L.~S. Blackford, A.~Petitet, R.~Pozo, K.~Remington, R.~C. Whaley,
  J.~Demmel, J.~Dongarra, I.~Duff, S.~Hammarling, G.~Henry, et~al.}, {\em An
  updated set of basic linear algebra subprograms ({BLAS})}, ACM Transactions
  on Mathematical Software, 28 (2002), pp.~135--151.

\bibitem{Boisvert1997}
{\sc R.~F. Boisvert, R.~Pozo, K.~Remington, R.~F. Barrett, and J.~J. Dongarra},
  {\em Matrix market: A web resource for test matrix collections}, in
  Proceedings of the IFIP TC2/WG2.5 Working Conference on Quality of Numerical
  Software: Assessment and Enhancement, London, UK, 1997, Chapman \& Hall,
  Ltd., pp.~125--137.

\bibitem{Braman2002}
{\sc K.~Braman, R.~Byers, and R.~Mathias}, {\em {The multishift QR algorithm.
  Part I: maintaining well-focused shifts and level 3 performance}}, SIAM
  Journal on Matrix Analysis and Applications, 23 (2002), pp.~929--947.

\bibitem{Braman2002a}
{\sc K.~Braman, R.~Byers, and R.~Mathias}, {\em {The multishift QR algorithm.
  Part II: aggressive early deflation}}, SIAM Journal on Matrix Analysis and
  Applications, 23 (2002), pp.~948--973.

\bibitem{CampsThesis}
{\sc D.~Camps}, {\em {Pole swapping methods for the eigenvalue problem}}, PhD
  thesis, KU Leuven, 2019.

\bibitem{Camps2019}
{\sc D.~Camps, N.~Mastronardi, R.~Vandebril, and P.~Van~Dooren}, {\em Swapping
  $2 \times 2$ blocks in the {S}chur and generalized {S}chur form}, Journal of
  Computational and Applied Mathematics, 373 (2019), pp.~1--8.

\bibitem{camps2019rational}
{\sc D.~Camps, K.~Meerbergen, and R.~Vandebril}, {\em A rational {QZ} method},
  SIAM Journal on Matrix Analysis and Applications, 40 (2019), pp.~943--972.

\bibitem{CaVaWaMa20}
{\sc D.~Camps, R.~Vandebril, D.~S. Watkins, and T.~Mach}, {\em On pole-swapping
  algorithms for the eigenvalue problem}.
\newblock Submitted for publication, 2020.

\bibitem{Elman07}
{\sc H.~Elman, A.~Ramage, and D.~Silvester}, {\em Algorithm {866}: {IFISS}, a
  {M}atlab toolbox for modelling incompressible flow}, ACM Transactions on
  Mathematical Software, 33 (2007), pp.~2--14.

\bibitem{Elman14}
{\sc H.~Elman, A.~Ramage, and D.~Silvester}, {\em {IFISS}: A computational
  laboratory for investigating incompressible flow problems}, SIAM Review, 56
  (2014), pp.~261--273.

\bibitem{Fra62}
{\sc J.~G.~F. Francis}, {\em {The QR Transformation---Part 2}}, The Computer
  Journal, 4 (1962), pp.~332--345,
  \url{https://doi.org/10.1093/comjnl/4.4.332}.

\bibitem{Kagstrom1993}
{\sc B.~K{\aa}gstr{\"o}m}, {\em A direct method for reordering eigenvalues in
  the generalized real {Schur} form of a regular matrix pair {(A, B)}}, in
  Linear Algebra for Large Scale and Real-Time Applications, M.~S. Moonen,
  G.~H. Golub, and B.~L.~R. De~Moor, eds., Springer Netherlands, Dordrecht,
  1993, pp.~195--218, \url{https://doi.org/10.1007/978-94-015-8196-7_11}.

\bibitem{Kagstrom2007}
{\sc B.~K{\aa}gstr{\"o}m and D.~Kressner}, {\em {Multishift variants of the
  {QZ} algorithm with aggressive early deflation}}, SIAM Journal on Matrix
  Analysis and Applications, 29 (2007), pp.~199--227.

\bibitem{Kagstrom1996}
{\sc B.~K{\aa}gstr{\"{o}}m and P.~Poromaa}, {\em {Computing eigenspaces with
  specified eigenvalues of a regular matrix pair {(A, B)} and condition
  estimation: theory, algorithms and software}}, Numerical Algorithms, 12
  (1996), pp.~369--407.

\bibitem{KaKrLa14}
{\sc L.~Karlsson, D.~Kressner, and B.~Lang}, {\em Optimally packed chains of
  bulges in multishift {QR} algorithms}, ACM Transactions on Mathematical
  Software, 40 (2014).

\bibitem{Oberwolfach}
{\sc J.~G. Korvink and E.~B. Rudnyi}, {\em Oberwolfach benchmark collection},
  in Dimension Reduction of Large-Scale Systems, P.~Benner, D.~C. Sorensen, and
  V.~Mehrmann, eds., Berlin, Heidelberg, 2005, Springer Berlin Heidelberg,
  pp.~311--315.

\bibitem{MaStVaWa20}
{\sc T.~Mach, T.~Steel, R.~Vandebril, and D.~S. Watkins}, {\em Pole-swapping
  algorithms for alternating and palindromic eigenvalue problems}, Vietnam
  Journal of Mathematics,  (2020).

\bibitem{Moler1973}
{\sc C.~B. Moler and G.~W. Stewart}, {\em {An algorithm for generalized matrix
  eigenvalue problems}}, SIAM Journal on Numerical Analysis, 10 (1973),
  pp.~1--52.

\bibitem{So92}
{\sc D.~C. Sorensen}, {\em Implicit application of polynomial filters in a
  k-step {Arnoldi} method}, SIAM Journal on Matrix Analysis and Applications,
  13 (1992), pp.~357--385.

\bibitem{VanDooren1981}
{\sc P.~{Van Dooren}}, {\em {A generalized eigenvalue approach for solving
  Riccati equations}}, SIAM Journal on Statistical Computing, 2 (1981),
  pp.~121--135, \url{https://doi.org/10.1137/0902010}.

\bibitem{wang2014intel}
{\sc E.~Wang, Q.~Zhang, B.~Shen, G.~Zhang, X.~Lu, Q.~Wu, and Y.~Wang}, {\em
  Intel math kernel library}, in High-Performance Computing on the
  Intel{\textregistered} Xeon Phi™, Springer, 2014, pp.~167--188.

\bibitem{Watkins1996}
{\sc D.~S. Watkins}, {\em {The transmission of shifts and shift blurring in the
  QR algorithm}}, Linear Algebra and its Applications, 241-243 (1996),
  pp.~877--896.

\bibitem{b333}
{\sc D.~S. Watkins}, {\em The Matrix Eigenvalue Problem: {GR} and {Krylov}
  Subspace Methods}, SIAM, Philadelphia, USA, 2007.

\end{thebibliography}

\end{document}